\documentclass[12pt]{amsart}

\usepackage[a4paper,top=2.5cm,bottom=1.9cm,left=2.5cm,right=2.5cm]{geometry}

\usepackage{amsmath}
\usepackage{amssymb}
\usepackage{amsthm}
\usepackage{enumerate}
\usepackage{mathtools}
\usepackage{latexsym}
\usepackage{algpseudocode}
\usepackage[ruled]{algorithm}
\usepackage{graphicx}
\usepackage{tikz}
\usepackage{caption}
\usepackage{subfig}
\usepackage{float}
\usepackage[latin1]{inputenc}
\usepackage[english]{babel}

\usepackage[colorlinks,hyperfigures]{hyperref}

\algrenewcommand{\algorithmicrequire}{\textbf{Input:}}
\algrenewcommand{\algorithmicensure}{\textbf{Output:}}

\newtheorem{theorem}{Theorem}[section]
\newtheorem{corollary}[theorem]{Corollary}
\newtheorem{proposition}[theorem]{Proposition}
\newtheorem{lemma}[theorem]{Lemma}

\theoremstyle{definition}
\newtheorem{definition}[theorem]{Definition}
\newtheorem{example}[theorem]{Example}
\newtheorem{remark}[theorem]{Remark}

\numberwithin{equation}{section}

\renewcommand{\leq}{\leqslant}
\renewcommand{\geq}{\geqslant}
\begin{document}

\title[Finite O-sequences and aCM genera]{A combinatorial description of\\ finite O-sequences and aCM genera}

\author[F.~Cioffi]{Francesca Cioffi}
\address{Francesca Cioffi\\ Universit\`a degli Studi di Napoli \lq\lq Federico II\rq\rq \\ Dipartimento di Matematica e Applicazioni \\ Complesso univ.~di Monte S.~Angelo, Via Cintia \\ 80126 Napoli \\ Italy.}
\email{\href{mailto:francesca.cioffi@unina.it}{francesca.cioffi@unina.it}}
\urladdr{\url{https://www.docenti.unina.it/francesca.cioffi}}

\author[P.~Lella]{Paolo Lella}
\address{Paolo Lella\\ Universit\`a degli Studi di Trento \\ Dipartimento di Matematica \\ Via Sommarive 14 \\ 38123 Povo (Trento) \\ Italy.}
\email{\href{mailto:math@paololella.it}{math@paololella.it}}
\urladdr{\url{http://www.paololella.it/}}

\author[M.~G.~Marinari]{Maria Grazia Marinari}
\address{Maria Grazia Marinari\\ Universit\`a degli Studi di Genova \\ Dipartimento di Matematica\\ Via Dodecaneso 35\\ 16146 Genova \\ Italy.}
\email{\href{mailto:marinari@dima.unige.it}{marinari@dima.unige.it}}
\urladdr{\url{http://www.dima.unige.it/~marinari/}}

\thanks{We thank Margherita Roggero for useful discussions about a previous version of this paper. The first author was partially supported by the PRIN 2010-11  \emph{Geometria delle variet\`a algebriche}, cofinanced by MIUR (Italy), and by GNSAGA. The second author was partially supported by PRIN 2010-11 \emph{Geometria delle variet\`a algebriche}, cofinanced by MIUR (Italy), by FIRB 2012 \emph{Moduli spaces and Applications} and by GNSAGA}

\begin{abstract}
The goal of this paper is to explicitly detect all the arithmetic genera of arithmetically Cohen-Macaulay projective curves with a given degree $d$. 
It is well-known that the arithmetic genus $g$ of a curve $C$ can be easily deduced from the $h$-vector of the curve; in the case where $C$ is arithmetically Cohen-Macaulay of degree $d$, $g$ must belong to the range of integers $\big\{0,\ldots,\binom{d-1}{2}\big\}$. 
We develop an algorithmic procedure that allows one to avoid constructing most of the possible $h$-vectors of $C$. The essential 
tools are a combinatorial description of the finite O-sequences of multiplicity $d$, and a sort of continuity result regarding the generation of the genera. 
The efficiency of our method is supported by computational evidence. As a consequence, we single out the minimal possible Castelnuovo-Mumford regularity of a curve with Cohen-Macaulay postulation and given degree and genus.
\end{abstract}

\maketitle

\section*{Introduction}

In this paper we introduce an algorithmic approach to the search of all possible arithmetic genera of an arithmetically Cohen-Macaulay (aCM for short) projective curve of given degree $d$. 
This problem has been studied in several instances, such as \cite[Example 4.6]{R}, and it has a role in the classification of algebraic curves, see 
for example \cite{H94,Nagel03} and the references therein.

The arithmetic genus $g$ of a curve appears in the constant term of the curve's Hilbert polynomial, hence it is related to the more general study of the coefficients of 
Hilbert polynomials (see \cite{H66} for a geometrical point of view, and \cite{ERV} in the context of local algebra). 

In fact, not only does the $h$-vector encode a lot of information about the geometry of the curve; the arithmetic genus of the curve is also easily deduced from it (\cite[Exercises 8.11 and 8.12]{H10}, \cite[Section 1.4]{Mi}). For an aCM projective scheme the $h$-vector is actually the Hilbert function of its artinian reduction. This result is mainly due to the fundamental 
paper of \cite{Ma} characterizing the Hilbert functions of standard graded algebras. 

We stress the fact that the work of Macaulay does not provide an algorithmic solution for the problem of deciding whether or not an aCM curve of degree $d$ and genus $g$ exists. This 
remark has been the starting point of our paper. By investigating the set of finite O-sequences of multiplicity $d$ and its properties we obtain our solution, 
both computational and theoretical, that relies on some closed formula considerably reducing the amount of real computations. We have not been able to find 
analogous results in literature. 
 
As a first step, we provide a very natural combinatorial description of finite O-sequences, by means of suitable connected graphs, and we obtain 
an efficient search algorithm of the arithmetic genera of Cohen-Macaulay curves (see Algorithm \ref{alg:genusTreeSearch} in Section \ref{sec:combinatorial}). 

Then, for every positive integer $d$, we denote by $R_d = \big[0,\binom{d-1}{2}\big] \cap \mathbb{N}$ the set of integers to which the genus of a Cohen-Macaulay curve of degree $d$ must belong, and we focus our attention on smaller ranges $R_d^s$, consisting of the genera of Cohen-Macaulay curves of degree $d$ and $h$-vector of length $s$. By introducing a convenient 
total ordering on the set of O-sequences of multiplicity $d$ and length $s$, we can single out each range $R_d^s$ (see Corollary \ref{cor:genusOrder}, Theorem \ref{th:degrevlex}, Propositions  \ref{dim procedura paolo}).

The integers in $R_d$ that can not be realized as genus of an aCM curve of degree $d$ are called {\em gaps}. Many of them are located outside every range $R_d^s$, some others lie {\em near} the maximal genus in $R_d^s$, for values of $s$ that can be exactly determined by suitable closed formulas (see Propositions \ref{easy gaps} and \ref{2}).

Finally, we provide an algorithm to compute all the genera of aCM curves for a given degree $d$, avoiding to construct all the corresponding O-sequences (see Algorithm \ref{alg:ACMgenera} in Section \ref{sec:computation}). The strategy supporting this algorithm combines the previous results together with a sort of {\em continuity} in the generation of the genera of aCM curves developed in Lemma \ref{gradi successivi} and applied in Theorem \ref{global continuity}. Experimental computations point out that only a small percentage of integers of $R_d$ needs to be checked by the search algorithm (see Tables \ref{tab:generaDistribution} and \ref{tab:computationTime}).

In Section \ref{sec:CMreg}, we apply our search algorithm to detect the minimal possible Castelnuovo-Mumford regularity of a curve with Cohen-Macaulay postulation and given degree and genus (Proposition \ref{regACM}). Moreover, we answer to a question posed in \cite{CDG} about the Castelnuovo-Mumford regularity of even dimensional projective subschemes having the same Hilbert function of a Cohen-Macaulay projective scheme (Example \ref{even aCM}).


\section{Generalities on O-sequences and aCM genera}

In this section, we state some notation and recall some basic results on O-sequences, referring to \cite{BH} and \cite{Va}. 
  
Given two positive integers $a$, $t$, the {\em binomial expansion of $a$ in base $t$} is the unique writing
\begin{equation}\label{expansion} 
 a = \tbinom{k(t)}{t} + \tbinom{k(t-1)}{t-1} + \cdots + \tbinom{k(j)}{j}
\end{equation} 
where $k(t)> k(t-1)>\cdots > k(j)\geq j\geq 1$ with the convention that $\binom{n}{m} = 0$ whenever $n<m$ and $\binom{n}{0}=1$ for every $n\geq 0$. Letting 
\[
a^{\langle t\rangle}:=\tbinom{k(t)+1}{t+1} + \tbinom{k(t-1)+1}{t} + \cdots + \tbinom{k(j)+1}{j+1},
\]
by an easy computation, one gets $(a+1)^{\langle t\rangle} > a^{\langle t\rangle}$.
A numerical function $\mathsf{h}:\mathbb N \rightarrow \mathbb N$ is {\it admissible} or an {\em O-sequence} if $\mathsf{h}(0)=1$ and $\mathsf{h}(t+1)\leq \mathsf{h}(t)^{\langle t\rangle}$ for every $t\geq 1$. 

If $\mathsf{h}$ is an admissible function and $\mathsf{h}(t)=0$ for some $t$, then $\mathsf{h}(t+i)=0$ for every $i>0$, and $\mathsf{h}$ is called a {\em finite or Artinian O-sequence}. For a finite O-sequence $(h_0,\ldots,h_{s-1})$ we assume $h_{s-1}\not=0$. The integer $s$ is the {\em length} of the O-sequence and the integer $e(\mathsf{h}):=\sum_{i=0}^{s-1}h_i$ is its multiplicity. 

It is well known that there is a bijective correspondence between the set of finite O-sequences of multiplicity $d$ and the set of Hilbert functions of a Cohen-Macaulay standard graded algebra of multiplicity $d$ over a field $K$ \cite[Theorem 1.5]{Va}. In fact, all these Hilbert functions can be computed from the finite O-sequences. In particular, if the graded algebra is the ring of regular functions on an aCM curve $C$ (i.e.~a closed subscheme $C\subset \mathbb P_K^n$ of dimension $1$), the Hilbert function $H_C$ of $C$ is the {\em $2$-th integral} of a finite O-sequences $\mathsf{h}=(h_0,h_1,\ldots,h_{s-1})$, i.e.~letting $H_C(0):=H_Z(0):=\mathsf{h}(0)=1$ and $H_Z(t)=H_Z(t-1)+\mathsf{h}(t)$ for every $t>0$, we have
$$H_C(t)=H_C(t-1)+H_Z(t), \text{ for every } t>0.$$
Hence, $\mathsf{h}$ is the so-called {\em $h$-vector} of $C$ and the Hilbert polynomial of $C$ is $p_C(z)=dz+1-g$ where, after an easy computation, we find that the {\em arithmetic genus} of $C$ is
\begin{equation}\label{genus}
g=1+(s-2)d-p(s-2)=\sum_{j=2}^{s-1} (j-1)h_j\geq 0. \qedhere
\end{equation} 
In this situation, we say that $H_C$ is an {\em aCM function} or a {\em Cohen-Macaulay postulation}, $p_C(z)$ is an {\em aCM polynomial} and $g$ is an {\em aCM genus}.

\begin{remark}\label{rem:prime osservazioni}
The following facts are immediate remarks:
\begin{enumerate}[(i)]
\item the arithmetic genus of an aCM curve is non-negative;
\item every positive integer $g$ is the genus of some aCM curve: it is enough to take any O-sequence $(1,h_1,g)$, with $h_1^{\langle 1 \rangle} \geq g$;
\item if $g$ is the arithmetic genus of some aCM curve $C_d$ of degree $d$, then there is also an aCM curve $C_{d+1}$ of degree $d+1$ with the same arithmetic genus $g$; indeed, if $\mathsf{h}=(1, h_1, h_2, \ldots, h_{s-1})$ is the $h$-vector of $C_d$, then the sequence $\mathsf{h}'=(1, h_1+1, h_2, \ldots, h_{s-1})$ is also an O-sequence and is the $h$-vector of a curve $C_{d+1}$ with Hilbert polynomial $(d+1)z+1-g$. Indeed, the multiplicity of the O-sequence $\mathsf{h}'$ is $d+1$ and then we apply formula \eqref{genus}, in which the integer $h_1$ does not occur. From a geometric point of view, this means that $C_{d+1}$ can be obtained as the union of $C_d$ and a line through a point of $C_d$. 
\end{enumerate}
\end{remark}


\section{A combinatorial description of finite O-sequences}\label{sec:combinatorial}

In this section, we consider a natural structure on the set of all finite O-sequences. This structure will entail 
both our search algorithm of the arithmetic genera of Cohen-Macaulay curves, and some useful information about the aCM genera, such as the existence of minimal genera corresponding to O-sequences with given length (and multiplicity).

We let $\mathsf{e}_i$ denote any sequence, of any length, consisting entirely of $0$ except $1$ in the $i$-th position. Moreover, we introduce the following compact notation for some particular sequences:
\[
(1^{\alpha_0},h_{i_1}^{\alpha_1},h_{i_2}^{\alpha_2},\ldots,h_{i_k}^{\alpha_k}) := (\underbrace{1,\ldots,1}_{\alpha_0\text{ times}},\underbrace{h_{i_1},\ldots,h_{i_1}}_{\alpha_1\text{ times}},\ldots,\underbrace{h_{i_k},\ldots,h_{i_k}}_{\alpha_k\text{ times}}).
\]

\begin{definition}
The \emph{O-sequences graph} is the directed graph $\mathcal{G}$ such that:
\begin{itemize}
\item the set of vertices $V(\mathcal{G})$ consists of the finite O-sequences;
\item the set of edges $E(\mathcal{G})$ consists of the pairs $(\mathsf{h},\mathsf{h}') \in V(\mathcal{G})^2$ s.t.~$\mathsf{h}'-\mathsf{h}=\mathsf{e}_i$ for some $i$  (i.e.~$(\mathsf{h},\mathsf{h}')\in E(\mathcal{G})$ if $\mathsf{h}'$ can be obtained from $\mathsf{h}$ by increasing by $1$ its $i$-th entry).
\end{itemize}
An edge $(\mathsf{h},\mathsf{h}')\in E(\mathcal{G})$ from $\mathsf{h}$ to $\mathsf{h}'$ is labeled by $\mathsf{e}_i$ if $\mathsf{h}' - \mathsf{h} = \mathsf{e}_i$. 
\end{definition}

Let us consider the map $g: \mathcal{G} \rightarrow \mathbb{N}$ that associates to each O-sequence the genus of an aCM curve having this O-sequence as $h$-vector. 

\begin{proposition}\label{whyOseqGraph}
The O-sequences graph $\mathcal{G}$ is a rooted connected graph without loops. The root is the O-sequence of multiplicity 1.
\end{proposition}
\begin{proof}
For any $\mathsf{h}=(1,h_1,\ldots,h_{s-1})$, the sequence $\mathsf{h}'=\mathsf{h}-\mathsf{e}_{s-1}$ is admissible so that there is an edge going from $\mathsf{h}'$ to $\mathsf{h}$. Repeating this procedure, we get the length one O-sequence $(1)$ which cannot be the head of any edge, proving that $\mathcal{G}$ is connected. There are no loops as each edge increases the multiplicity by $1$. 
\end{proof}

\begin{remark}
Denoted by $d_{\mathcal{G}}(\mathsf{h})$ the distance of the node $\mathsf{h}$ from the root, we have $d_{\mathcal{G}}(\mathsf{h}) = e(\mathsf{h})-1$.
\end{remark}

We are going to define a subgraph $\mathcal{T}\subset\mathcal{G}$ which will turn out to be a spanning tree. In this way, we can design ad hoc algorithms to visit the tree in order to quickly find the O-sequences with the properties we will look for. The idea for determining $\mathcal{T}$ is the one used in the proof of Proposition \ref{whyOseqGraph}. For each node of $\mathcal{G}$, we consider only the edge coming from the O-sequence obtained lowering by $1$ the value with the greatest index. Indeed, notice that each O-sequence $\mathsf{h}$ (of any length $s$) has a successor in $\mathcal{T}$, as $\mathsf{h}+\mathsf{e}_s$ is always a finite O-sequence, whereas the sequence $\mathsf{h}+\mathsf{e}_{s-1}$ might not be admissible.

\begin{definition}\label{def:tree}
We call \emph{O-sequences tree} the subgraph $\mathcal{T}\subset \mathcal{G}$ such that:
\begin{itemize}
\item $V(\mathcal{T}) = V(\mathcal{G})$;
\item $E(\mathcal{T}) = \left\{
(\mathsf{h},\mathsf{h}') \in E(\mathcal{G})\ \big\vert\  \mathsf{h}'= \mathsf{h}+\mathsf{e}_s  \text{ or } \mathsf{h}'= \mathsf{h}+\mathsf{e}_{s-1}, \text{ if } h_{s-2}^{\langle s-2 \rangle} > h_{s-1}\right\}$.
\end{itemize}
\end{definition}

\begin{figure}[!ht]
\begin{center}
\begin{tikzpicture}[>=latex,line join=bevel,scale=0.58]
 \node (d1) at (285bp,300bp) [inner sep=2pt] {\tiny $(1)$};

  \node (d2_11) at (285bp,240bp) [inner sep=2pt] {\tiny $(1^2)$};
  
  \node (d3_111) at (255bp,180bp) [inner sep=2pt] {\tiny $(1^3)$};
  \node (d3_12) at (315bp,180bp) [inner sep=2pt] {\tiny $(1,2)$};
  
  \node (d4_1111) at (200bp,120bp) [inner sep=2pt] {\tiny $(1^4)$};
  \node (d4_121) at (285bp,120bp) [inner sep=2pt] {\tiny $(1,2,1)$};
  \node (d4_13) at (370bp,120bp) [inner sep=2pt] {\tiny $(1,3)$};
  
  \node (d5_11111) at (135bp,60bp) [inner sep=2pt] {\tiny $(1^5)$};
  \node (d5_1211) at (210bp,60bp) [inner sep=2pt] {\tiny $(1,2,1^2)$};
  \node (d5_122) at (285bp,60bp) [inner sep=2pt] {\tiny $(1,2^2)$};
  \node (d5_131) at (360bp,60bp) [inner sep=2pt] {\tiny $(1,3,1)$};
  \node (d5_14) at (435bp,60bp) [inner sep=2pt] {\tiny $(1,4)$};
  
  \node (d6_111111) at (75bp,0bp) [inner sep=2pt] {\tiny $(1^6)$};
  \node (d6_12111) at (130bp,0bp) [inner sep=2pt] {\tiny $(1,2,1^3)$};
  \node (d6_1221) at (195bp,0bp) [inner sep=2pt] {\tiny $(1,2^2,1)$};
  \node (d6_1311) at (255bp,0bp) [inner sep=2pt] {\tiny $(1,3,1^2)$};
  \node (d6_123) at (315bp,0bp) [inner sep=2pt] {\tiny $(1,2,3)$};
  \node (d6_132) at (385bp,0bp) [inner sep=2pt] {\tiny $(1,3,2)$};
  \node (d6_15) at (495bp,0bp) [inner sep=2pt] {\tiny $(1,5)$};
  \node (d6_141) at (435bp,0bp) [inner sep=2pt] {\tiny $(1,4,1)$};
  
  \node (d7_1111111) at (-20bp,-60bp) [inner sep=2pt] {\tiny $(1^7)$};
  \node (d7_121111) at (25bp,-60bp) [inner sep=2pt] {\tiny $(1,2,1^4)$};
  \node (d7_12211) at (85bp,-60bp) [inner sep=2pt] {\tiny $(1,2^2,1^2)$};
  \node (d7_13111) at (145bp,-60bp) [inner sep=2pt] {\tiny $(1,3,1^3)$};
  \node (d7_1222) at (200bp,-60bp) [inner sep=2pt] {\tiny $(1,2^3)$};  
  \node (d7_1231) at (255bp,-60bp) [inner sep=2pt] {\tiny $(1,2,3,1)$};
  \node (d7_1321) at (315bp,-60bp) [inner sep=2pt] {\tiny $(1,3,2,1)$};
  \node (d7_133) at (435bp,-60bp) [inner sep=2pt] {\tiny $(1,3^2)$};
  \node (d7_1411) at (380bp,-60bp) [inner sep=2pt] {\tiny $(1,4,1^2)$};
  \node (d7_142) at (485bp,-60bp) [inner sep=2pt] {\tiny $(1,4,2)$};
  \node (d7_151) at (540bp,-60bp) [inner sep=2pt] {\tiny $(1,5,1)$};
  \node (d7_16) at (590bp,-60bp) [inner sep=2pt] {\tiny $(1,6)$};

  \draw [->,thin] (d6_111111) -- (d7_1111111);
  \draw [->,thin,dashed] (d6_111111) -- (d7_121111);
  \draw [->,thin] (d6_12111) -- (d7_121111);
  \draw [->,thin,dashed] (d6_12111) -- (d7_12211);
  \draw [->,thin,dashed] (d6_12111) -- (d7_13111);
  \draw [->,thin] (d6_1221) -- (d7_12211);
  \draw [->,thin] (d6_1221) -- (d7_1222);
  \draw [->,thin,dashed] (d6_1221) -- (d7_1231);
  \draw [->,thin,dashed] (d6_1221) -- (d7_1321);
  \draw [->,thin] (d6_123) -- (d7_1231);
  \draw [->,thin,dashed] (d6_123) -- (d7_133);
  \draw [->,thin] (d6_1311) -- (d7_13111);
  \draw [->,thin,dashed] (d6_1311) -- (d7_1321);
  \draw [->,thin,dashed] (d6_1311) -- (d7_1411);
  \draw [->,thin] (d6_132) -- (d7_1321);
  \draw [->,thin] (d6_132) -- (d7_133);
  \draw [->,thin,dashed] (d6_132) -- (d7_142);
  \draw [->,thin] (d6_141) -- (d7_1411);
  \draw [->,thin] (d6_141) -- (d7_142);
  \draw [->,thin,dashed] (d6_141) -- (d7_151);
  \draw [->,thin] (d6_15) -- (d7_151);
  \draw [->,thin] (d6_15) -- (d7_16);
  
  \draw [->,thin] (d5_122) -- (d6_123);
  \draw [->,thin] (d5_131) -- (d6_1311);
  \draw [->,thin,dashed] (d5_1211) -- (d6_1311);
  \draw [->,thin,dashed] (d3_111) -- (d4_121);
  \draw [->,thin,dashed] (d5_131) -- (d6_141);
  \draw [->,thin] (d2_11) -- (d3_111);
  \draw [->,thin] (d4_13) -- (d5_131);
  \draw [->,thin] (d5_131) -- (d6_132);
  \draw [->,thin] (d5_11111) -- (d6_111111);
  \draw [->,thin] (d1) -- (d2_11);
  \draw [->,thin,dashed] (d5_11111) -- (d6_12111);
  \draw [->,thin] (d5_122) -- (d6_1221);
  \draw [->,thin] (d4_121) -- (d5_122);
  \draw [->,thin] (d3_111) -- (d4_1111);
  \draw [->,thin] (d3_12) -- (d4_13);
  \draw [->,thin] (d5_14) -- (d6_15);
  \draw [->,thin,dashed] (d5_1211) -- (d6_1221);
  \draw [->,thin] (d4_1111) -- (d5_11111);
  \draw [->,thin] (d3_12) -- (d4_121);
  \draw [->,thin] (d5_14) -- (d6_141);
  \draw [->,thin] (d4_13) -- (d5_14);
  \draw [->,thin] (d4_121) -- (d5_1211);
  \draw [->,thin,dashed] (d5_122) -- (d6_132);
  \draw [->,thin] (d2_11) -- (d3_12);
  \draw [->,thin,dashed] (d4_121) -- (d5_131);
  \draw [->,thin,dashed] (d4_1111) -- (d5_1211);
  \draw [->,thin] (d5_1211) -- (d6_12111);
\end{tikzpicture}
\end{center}
\caption{\label{fig:OseqTree} The O-sequence graph $\mathcal{G}$ up to multiplicity $7$. The dashed edges are edges of $\mathcal{G}$ that do not belong to the spanning tree $\mathcal{T}$.}
\end{figure}
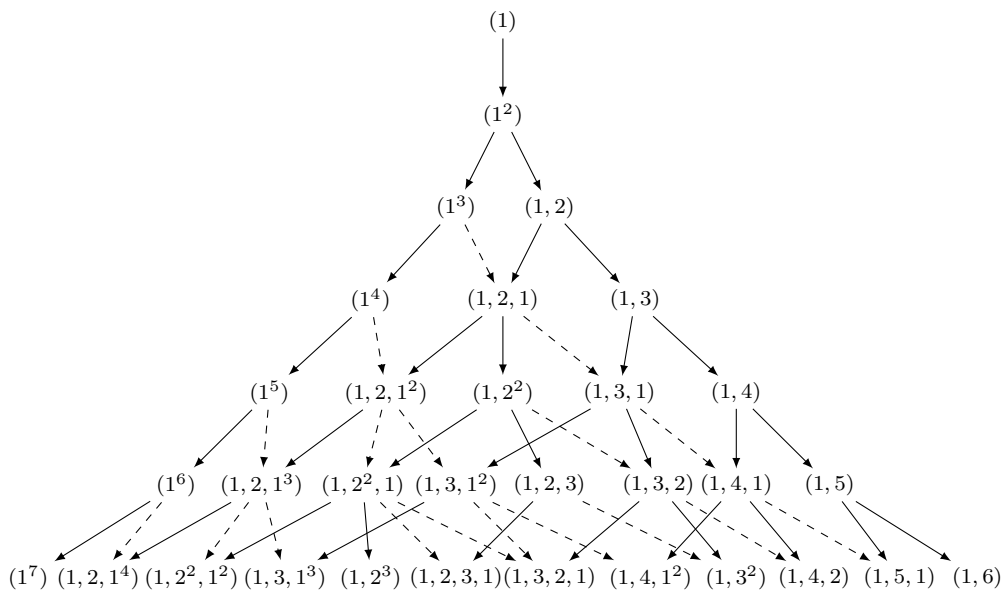

In most situations, we will work with O-sequences with given multiplicity (i.e.~with nodes of $\mathcal{G}$ at the same distance from the root) or with given length. We denote by $\mathcal{G}_d$ the set of O-sequences of multiplicity $d$ and by $\mathcal{G}^s$ the set of O-sequences of length $s$. 

\begin{remark}
As in the spanning tree $\mathcal{T}$ each vertex is the tail of at most 2 edges, we have that $\vert\mathcal{G}_d\vert < 2\vert\mathcal{G}_{d-1}\vert$. Moreover, since $\vert\mathcal{G}_2\vert = 1$, by recursion $\vert\mathcal{G}_d\vert < 2^{d-2}$.
\end{remark}

\begin{proposition}
The subgraph $\mathcal{G}^{s} \subset \mathcal{G}$ is a rooted connected graph with root $(1^s)$ containing a spanning tree $\mathcal{T}^s$ with the same root.
\end{proposition}

\begin{proof}
We need to show that, for any O-sequence $\mathsf{h} \neq (1^s)$ of length $s$, there exists another O-sequence of the same length with multiplicity $e(\mathsf{h})-1$. If $k = \max\{1\leq i \leq s-1\ \vert\ h_i > 1\}$, then $\mathsf{h} = (1,h_1,\ldots,h_k,1^{s-k-1})$ and $\mathsf{h}'= (1,h_1,\ldots,h_k-1,1^{s-k-1})$ is admissible.
\end{proof}

\begin{remark}
Denoted by $d^s_{\mathcal{G}}(\mathsf{h})$ the distance of the node $\mathsf{h}$ from the root of $\mathcal{G}^s$, we have $d^s_{\mathcal{G}}(\mathsf{h}) = d_{\mathcal{G}}(\mathsf{h}) - (s-1) = e(\mathsf{h})-s$.
\end{remark}

\begin{figure}[!ht]
\begin{center}
\begin{tikzpicture}[>=latex,line join=bevel,scale=0.58]
  \node (d1) at (340bp,300bp) [inner sep=2pt] {\tiny $(1)$};

  \draw [very thin] (290bp,290bp) --node[above]{\small $\mathcal{G}^1$} (350bp,350bp) -- (390bp,310bp) -- (330bp,250bp);
  
  \node (d2_11) at (280bp,240bp) [inner sep=2pt] {\tiny $(1^2)$};
  
   \draw [very thin] (545bp,-70bp) -- (235bp,235bp) --node[above]{\small $\mathcal{G}^2$} (290bp,290bp) -- (600bp,-20bp);

  \node (d3_111) at (220bp,180bp) [inner sep=2pt] {\tiny $(1^3)$};
  \node (d3_12) at (340bp,180bp) [inner sep=2pt] {\tiny $(1,2)$};
  
  \draw [very thin] (380bp,-70bp) -- (170bp,170bp) --node[above]{\small $\mathcal{G}^3$} (235bp,235bp);
  
  \node (d4_1111) at (160bp,120bp) [inner sep=2pt] {\tiny $(1^4)$};
  \node (d4_121) at (285bp,120bp) [inner sep=2pt] {\tiny $(1,2,1)$};
  \node (d4_13) at (400bp,120bp) [inner sep=2pt] {\tiny $(1,3)$};
  
  \draw [very thin] (162bp,-70bp) -- (125bp,125bp) --node[above]{\small $\mathcal{G}^4$} (170bp,170bp);
  
  \node (d5_11111) at (100bp,60bp) [inner sep=2pt] {\tiny $(1^5)$};
  \node (d5_1211) at (190bp,60bp) [inner sep=2pt] {\tiny $(1,2,1^2)$};
  \node (d5_122) at (300bp,60bp) [inner sep=2pt] {\tiny $(1,2^2)$};
  \node (d5_131) at (360bp,60bp) [inner sep=2pt] {\tiny $(1,3,1)$};
  \node (d5_14) at (460bp,60bp) [inner sep=2pt] {\tiny $(1,4)$};
 
   \draw [very thin] (49bp,-70bp) -- (80bp,80bp) --node[above]{\small $\mathcal{G}^5$} (125bp,125bp);
  
  \node (d6_111111) at (40bp,0bp) [inner sep=2pt] {\tiny $(1^6)$};
  \node (d6_12111) at (100bp,0bp) [inner sep=2pt] {\tiny $(1,2,1^3)$};
  \node (d6_1221) at (195bp,0bp) [inner sep=2pt] {\tiny $(1,2^2,1)$};
  \node (d6_1311) at (265bp,0bp) [inner sep=2pt] {\tiny $(1,3,1^2)$};
  \node (d6_123) at (352bp,0bp) [inner sep=2pt] {\tiny $(1,2,3)$};
  \node (d6_132) at (398bp,0bp) [inner sep=2pt] {\tiny $(1,3,2)$};
  \node (d6_141) at (443bp,0bp) [inner sep=2pt] {\tiny $(1,4,1)$};
  \node (d6_15) at (520bp,0bp) [inner sep=2pt] {\tiny $(1,5)$};
  
     \draw [very thin] (-6bp,-70bp) -- (20bp,20bp) --node[above]{\small $\mathcal{G}^6$} (80bp,80bp);

  \node (d7_1111111) at (-20bp,-60bp) [inner sep=2pt] {\tiny $(1^7)$};
  \node (d7_121111) at (24bp,-60bp) [inner sep=2pt] {\tiny $(1,2,1^4)$};
  \node (d7_12211) at (81bp,-60bp) [inner sep=2pt] {\tiny $(1,2^2,1^2)$};
  \node (d7_13111) at (134bp,-60bp) [inner sep=2pt] {\tiny $(1,3,1^3)$};
  \node (d7_1222) at (182bp,-60bp) [inner sep=2pt] {\tiny $(1,2^3)$};  
  \node (d7_1231) at (229bp,-60bp) [inner sep=2pt] {\tiny $(1,2,3,1)$};
  \node (d7_1321) at (285bp,-60bp) [inner sep=2pt] {\tiny $(1,3,2,1)$};
  \node (d7_133) at (400bp,-60bp) [inner sep=2pt] {\tiny $(1,3^2)$};
  \node (d7_1411) at (338bp,-60bp) [inner sep=2pt] {\tiny $(1,4,1^2)$};
  \node (d7_142) at (450bp,-60bp) [inner sep=2pt] {\tiny $(1,4,2)$};
  \node (d7_151) at (500bp,-60bp) [inner sep=2pt] {\tiny $(1,5,1)$};
  \node (d7_16) at (580bp,-60bp) [inner sep=2pt] {\tiny $(1,6)$};
  
   \draw [very thin] (-45bp,-70bp) -- (-15bp,-15bp) --node[above]{\small $\mathcal{G}^7$} (20bp,20bp);

  \draw [->,thick,dotted,black!25] (d6_111111) -- (d7_1111111);
  \draw [->,thin] (d6_111111) -- (d7_121111);
  \draw [->,thick,dotted,black!25] (d6_12111) -- (d7_121111);
  \draw [->,thin] (d6_12111) -- (d7_12211);
  \draw [->,thin] (d6_12111) -- (d7_13111);
  \draw [->,thick,dotted,black!25] (d6_1221) -- (d7_12211);
  \draw [->,thin] (d6_1221) -- (d7_1222);
  \draw [->,thin] (d6_1221) -- (d7_1231);
  \draw [->,thin,dashed] (d6_1221) -- (d7_1321);
  \draw [->,thick,dotted,black!25] (d6_123) -- (d7_1231);
  \draw [->,thin,dashed] (d6_123) -- (d7_133);
  \draw [->,thick,dotted,black!25] (d6_1311) -- (d7_13111);
  \draw [->,thin] (d6_1311) -- (d7_1321);
  \draw [->,thin] (d6_1311) -- (d7_1411);
  \draw [->,thick,dotted,black!25] (d6_132) -- (d7_1321);
  \draw [->,thin] (d6_132) -- (d7_133);
  \draw [->,thin,dashed] (d6_132) -- (d7_142);
  \draw [->,thick,dotted,black!25] (d6_141) -- (d7_1411);
  \draw [->,thin] (d6_141) -- (d7_142);
  \draw [->,thin] (d6_141) -- (d7_151);
  \draw [->,thick,dotted,black!25] (d6_15) -- (d7_151);
  \draw [->,thin] (d6_15) -- (d7_16);

  \draw [->,thin] (d5_122) -- (d6_123);
  \draw [->,thick,dotted,black!25] (d5_131) -- (d6_1311);
  \draw [->,thin] (d5_1211) -- (d6_1311);
  \draw [->,thin] (d3_111) -- (d4_121);
  \draw [->,thin] (d5_131) -- (d6_141);
  \draw [->,thick,dotted,black!25] (d2_11) -- (d3_111);
  \draw [->,thick,dotted,black!25] (d4_13) -- (d5_131);
  \draw [->,thin] (d5_131) -- (d6_132);
  \draw [->,thick,dotted,black!25] (d5_11111) -- (d6_111111);
  \draw [->,thick,dotted,black!25] (d1) -- (d2_11);
  \draw [->,thin] (d5_11111) -- (d6_12111);
  \draw [->,thick,dotted,black!25] (d5_122) -- (d6_1221);
  \draw [->,thin] (d4_121) -- (d5_122);
  \draw [->,thick,dotted,black!25] (d3_111) -- (d4_1111);
  \draw [->,thin] (d3_12) -- (d4_13);
  \draw [->,thin] (d5_14) -- (d6_15);
  \draw [->,thin] (d5_1211) -- (d6_1221);
  \draw [->,thick,dotted,black!25] (d4_1111) -- (d5_11111);
  \draw [->,thick,dotted,black!25] (d3_12) -- (d4_121);
  \draw [->,thick,dotted,black!25] (d5_14) -- (d6_141);
  \draw [->,thin] (d4_13) -- (d5_14);
  \draw [->,thick,dotted,black!25] (d4_121) -- (d5_1211);
  \draw [->,thin,dashed] (d5_122) -- (d6_132);
  \draw [->,thin] (d2_11) -- (d3_12);
  \draw [->,thin] (d4_121) -- (d5_131);
  \draw [->,thin] (d4_1111) -- (d5_1211);
  \draw [->,thick,dotted,black!25] (d5_1211) -- (d6_12111);
\end{tikzpicture}
\caption{\label{fig:OseqGraph2} The subgraphs $\mathcal G^s$ of the O-sequence graph with given length $s$. Along the grey dotted edges the length increases, so such edges of $\mathcal{G}$ do not belong to any subgraph $\mathcal G^s$. The dashed edges are edges of $\mathcal G^s$ that do not belong to the corresponding spanning tree $\mathcal{T}^s$.}
\end{center}
\end{figure}
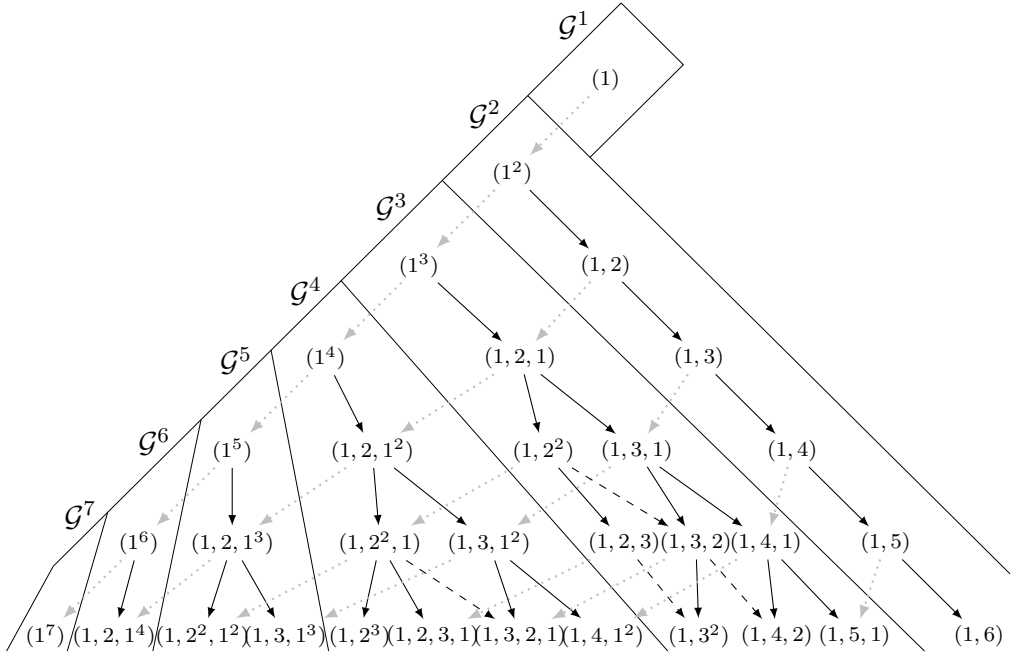

$\mathcal{G}_d$ is not a subgraph of $\mathcal{G}$, as there are no edges of $\mathcal{G}$ between O-sequences with the same multiplicity. But the edges of $\mathcal{G}$ induce the following natural partial order on $\mathcal{G}_d$.

\begin{definition}\label{ordering}
Two O-sequences $\mathsf{h}_1$ and $\mathsf{h}_2$ in $\mathcal{G}_d$ are {\em directly comparable} if there exists $\mathsf{h}_0 \in \mathcal{G}_{d-1}$ such that $\mathsf{h}_1 = \mathsf{h}_0 + \mathsf{e}_i$ and $\mathsf{h}_2 = \mathsf{h}_0 + \mathsf{e}_j$, i.e.~$\mathsf{h}_1 - \mathsf{h}_2 = \mathsf{e}_i-\mathsf{e}_j$. On directly comparable O-sequences we consider the order
\begin{equation}\label{eq:partial ordering}
\mathsf{h}_1 \prec \mathsf{h}_2\quad\Longleftrightarrow\quad i < j
\end{equation}
and denote by $\prec$ also its transitive closure in $\mathcal{G}_d$.
\end{definition}

The partial order $\prec$ gives a natural structure of directed graph to $\mathcal{G}_d$. The edges are all the possible pairs $(\mathsf{h},\mathsf{h}') \in V(\mathcal{G}_d)^2$ such that $\mathsf{h}= \mathsf{h}'+\mathsf{e}_j-\mathsf{e}_i$ and $j>i$ (see Figure \ref{fig:OseqGraph3}). As before, we define a spanning tree of the graph structure of $\mathcal{G}_d$ which allows us to efficiently examine the set of O-sequences with given multiplicity. The same procedure is also extended to the set of O-sequences $\mathcal{G}_d^s$ with given multiplicity $d$ and length $s$. Moreover, we let 
\begin{equation}\label{eq:min}
{\mathsf h}^s(d):=(1,d-s+1,1^{s-2})\quad \text{ and }\quad {\mathsf g}^s(d):=g({\mathsf h}^s(d))=\tbinom{s-1}{2}.
\end{equation}

\begin{proposition}\label{prop:degreeLength}
\begin{enumerate}[(i)]
\item\label{it:degreeD} The graph $\mathcal{G}_d$ contains a spanning tree $\mathcal{T}_d$ with root the O-sequence $(1,d-1)$.
\item\label{it:degreeDlengthS} The subgraph $\mathcal{G}_d^s$ contains a spanning tree $\mathcal{T}^s_d$ with root the O-sequence ${\mathsf h}^s(d)$. Thus, $\mathcal{G}_d^s$ is also connected. 
\end{enumerate}
\end{proposition}

\begin{proof}
\emph{(\ref{it:degreeD})} For each vertex $\mathsf{h} \in \mathcal{G}_d\setminus \{(1,d-1)\}$, the spanning tree $\mathcal{T}_d$ contains the edge $\mathsf{e}_{s-1}- \mathsf{e}_1$ going from $\mathsf{h}'=\mathsf{h}-\mathsf{e}_{s-1}+ \mathsf{e}_1$ to $\mathsf{h}$, where $s$ is the length of $\mathsf{h}$.

\emph{(\ref{it:degreeDlengthS})} For each vertex $\mathsf{h} = (1,h_1,\ldots,h_i,1^{d-\sum_{j=0}^i h_j}) \in \mathcal{G}_d^s \setminus \{(1,d-s+1,1^{s-2})\}$ (i.e.~$i>1$), the spanning tree $\mathcal{T}_d^s$ contains the edge $\mathsf{e}_i - \mathsf{e}_1$ going from $\mathsf{h}' = \mathsf{h} - \mathsf{e}_i + \mathsf{e}_{1}$ to $\mathsf{h}$. \end{proof}

\begin{corollary}\label{cor:genusOrder}
The order induced on $\mathcal{G}_d$ by the total order on $\mathbb{N}$ through the map $g:~\mathcal{G}_d \rightarrow \mathbb{N}$ is a refinement of the partial order $\prec$. In particular, ${\mathsf h}^s(d) = \min({\mathcal G}_d^s)$ with respect to $\prec$, ${\mathsf g}^s(d)$ is the minimal genus corresponding to an O-sequence of length $s$ and multiplicity $d$ and it does not depend on $d$.
\end{corollary}

\begin{proof}
If $\mathsf{h}_1 - \mathsf{h}_2 = \mathsf{e}_i - \mathsf{e}_j$, then $g(\mathsf{h}_1) = g(\mathsf{h}_2) + (i-1)-(j-1)=g(\mathsf{h}_2)+i-j$, by formula \eqref{genus}. Hence, we obtain 
\[
\mathsf{h}_1 \prec \mathsf{h}_2 \quad\Longleftrightarrow\quad i < j\quad\Longrightarrow\quad g(\mathsf{h}_1) < g(\mathsf{h}_2) 
\]
and the assertion about the minimum follows by Proposition \ref{prop:degreeLength}.
\end{proof}

As the minimal genus ${\mathsf g}^s(d)$ does not depend on the value of $d$, from now on we will simply denote it by ${\mathsf g}^s$.

\begin{figure}[!ht]
\begin{center}
\begin{tikzpicture}[>=latex,line join=bevel,scale=0.53]
  \node (d1) at (285bp,300bp) [inner sep=1.5pt] {\tiny $(1)$};

  \draw [very thin] (185bp,280bp) -- (315bp,280bp) -- (385bp,320bp)node[above left]{\small $\mathcal{G}_1$} -- (255bp,320bp) -- cycle;
  
    \node (d2_11) at (285bp,240bp) [inner sep=1.5pt]  {\tiny $(1^2)$};
    
  \draw [very thin] (185bp,220bp) -- (315bp,220bp) -- (385bp,260bp)node[above left]{\small $\mathcal{G}_2$} -- (255bp,260bp) -- cycle;

  \node (d3_111) at (250bp,180bp) [inner sep=1.5pt]  {\tiny $(1^3)$};
  \node (d3_12) at (320bp,180bp) [inner sep=1.5pt]  {\tiny $(1,2)$};

  \draw [<-,thin] (d3_111) -- (d3_12);
  
   \draw [very thin] (150bp,160bp) -- (350bp,160bp) -- (420bp,200bp)node[above left]{\small $\mathcal{G}_3$} -- (220bp,200bp) -- cycle;

  \node (d4_1111) at (200bp,120bp) [inner sep=1.5pt] {\tiny $(1^4)$};
  \node (d4_121) at (285bp,120bp) [inner sep=1.5pt] {\tiny $(1,2,1)$};
  \node (d4_13) at (370bp,120bp) [inner sep=1.5pt] {\tiny $(1,3)$};
 
  \draw [<-,thin] (d4_1111) -- (d4_121);
  \draw [<-,thin] (d4_121) -- (d4_13);

     \draw [very thin] (100bp,100bp) -- (400bp,100bp) -- (470bp,140bp)node[above left]{\small $\mathcal{G}_4$} -- (170bp,140bp) -- cycle;

  \node (d5_11111) at (140bp,20bp) [inner sep=1.5pt] {\tiny $(1^5)$};
  \node (d5_1211) at (200bp,20bp) [inner sep=1.5pt] {\tiny $(1,2,1^2)$};
  \node (d5_122) at (270bp,60bp) [inner sep=1.5pt] {\tiny $(1,2^2)$};
  \node (d5_131) at (340bp,60bp) [inner sep=1.5pt] {\tiny $(1,3,1)$};
  \node (d5_14) at (410bp,60bp) [inner sep=1.5pt] {\tiny $(1,4)$};
  
  \draw [<-,thin] (d5_11111) -- (d5_1211);
  \draw [<-,thin] (d5_1211) -- (d5_122);
  \draw [<-,thin] (d5_1211) -- (d5_131);
  \draw [<-,thin] (d5_122) -- (d5_131);
  \draw [<-,thin] (d5_131) -- (d5_14);

    \draw [very thin] (40bp,0bp) -- (370bp,0bp) -- (510bp,80bp)node[above left]{\small $\mathcal{G}_5$} -- (180bp,80bp) -- cycle;

  \node (d6_111111) at (80bp,-100bp) [inner sep=1.5pt] {\tiny $(1^6)$};
  \node (d6_12111) at (160bp,-100bp) [inner sep=1.5pt] {\tiny $(1,2,1^3)$};
  \node (d6_1221) at (200bp,-70bp) [inner sep=1.5pt] {\tiny $(1,2^2,1)$};  
  \node (d6_1311) at (280bp,-70bp) [inner sep=1.5pt] {\tiny $(1,3,1^2)$};
  \node (d6_123) at (240bp,-40bp) [inner sep=1.5pt] {\tiny $(1,2,3)$};
  \node (d6_132) at (320bp,-40bp) [inner sep=1.5pt] {\tiny $(1,3,2)$};
  \node (d6_141) at (400bp,-40bp) [inner sep=1.5pt] {\tiny $(1,4,1)$};
  \node (d6_15) at (480bp,-40bp) [inner sep=1.5pt] {\tiny $(1,5)$};
  
  \draw [<-,thin] (d6_12111) -- (d6_1221);
  \draw [<-,thin] (d6_12111) -- (d6_1311);
  \draw [<-,thin] (d6_111111) -- (d6_12111);
  \draw [<-,thin] (d6_1221) -- (d6_123);
  \draw [<-,thin] (d6_1221) -- (d6_132);
  \draw [<-,thin] (d6_1221) -- (d6_1311);
  \draw [<-,thin] (d6_1311) -- (d6_141);
  \draw [<-,thin] (d6_1311) -- (d6_132);
  \draw [<-,thin] (d6_123) -- (d6_132);
  \draw [<-,thin] (d6_132) -- (d6_141);
  \draw [<-,thin] (d6_141) -- (d6_15);

      \draw [very thin] (-20bp,-120bp) -- (405bp,-120bp) -- (580bp,-20bp)node[above left]{\small $\mathcal{G}_6$} -- (155bp,-20bp) -- cycle;

  \node (d7_1111111) at (30bp,-250bp) [inner sep=1.5pt] {\tiny $(1^7)$};
  \node (d7_121111) at (100bp,-250bp) [inner sep=1.5pt] {\tiny $(1,2,1^4)$};
  \node (d7_12211) at (140bp,-220bp) [inner sep=1.5pt] {\tiny $(1,2^2,1^2)$};
  \node (d7_1222) at (110bp,-190bp) [inner sep=1.5pt] {\tiny $(1,2^3)$};  
  \node (d7_1231) at (210bp,-160bp) [inner sep=1.5pt] {\tiny $(1,2,3,1)$};
  \node (d7_13111) at (215bp,-220bp) [inner sep=1.5pt] {\tiny $(1,3,1^3)$};
  \node (d7_1321) at (260bp,-190bp) [inner sep=1.5pt] {\tiny $(1,3,2,1)$};
  \node (d7_133) at (300bp,-160bp) [inner sep=1.5pt] {\tiny $(1,3^2)$};
  \node (d7_1411) at (340bp,-190bp) [inner sep=1.5pt] {\tiny $(1,4,1^2)$};
  \node (d7_142) at (380bp,-160bp) [inner sep=1.5pt] {\tiny $(1,4,2)$};
  \node (d7_151) at (460bp,-160bp) [inner sep=1.5pt] {\tiny $(1,5,1)$};
  \node (d7_16) at (540bp,-160bp) [inner sep=1.5pt] {\tiny $(1,6)$};
  
  \draw [very thin] (-70bp,-270bp) -- (430bp,-270bp) -- (640bp,-140bp)node[above left]{\small $\mathcal{G}_7$} -- (140bp,-140bp) -- cycle;

  \draw [<-,thin] (d7_1111111) -- (d7_121111);
  \draw [<-,thin] (d7_121111) -- (d7_12211);
  \draw [<-,thin] (d7_121111) -- (d7_13111);
  \draw [<-,thin] (d7_12211) -- (d7_13111);
  \draw [<-,thin] (d7_12211) -- (d7_1222);
  \draw [<-,thin] (d7_12211) -- (d7_1321);
  \draw [<-,thin] (d7_12211) -- (d7_1231);
  \draw [<-,thin] (d7_1222) -- (d7_1321);
  \draw [<-,thin] (d7_1222) -- (d7_1231);
  \draw [<-,thin] (d7_1231) -- (d7_1321);
  \draw [<-,thin] (d7_1231) -- (d7_133);
  \draw [<-,thin] (d7_13111) -- (d7_1321);
  \draw [<-,thin] (d7_13111) -- (d7_1411);
  \draw [<-,thin] (d7_1321) -- (d7_133);
  \draw [<-,thin] (d7_1321) -- (d7_1411);
  \draw [<-,thin] (d7_1321) -- (d7_142);
  \draw [<-,thin] (d7_1411) -- (d7_142);
  \draw [<-,thin] (d7_1411) -- (d7_151);
  \draw [<-,thin] (d7_151) -- (d7_16);
  \draw [<-,thin] (d7_142) -- (d7_151);
  \draw [<-,thin] (d7_133) -- (d7_142);

  \draw [-,thin,dotted] (d6_111111) -- (d7_1111111);
  \draw [-,thin,dotted] (d6_111111) -- (d7_121111);
  \draw [-,thin,dotted] (d6_12111) -- (d7_121111);
  \draw [-,thin,dotted] (d6_12111) -- (d7_12211);
  \draw [-,thin,dotted] (d6_12111) -- (d7_13111);
  \draw [-,thin,dotted] (d6_1221) -- (d7_12211);
  \draw [-,thin,dotted] (d6_1221) -- (d7_1222);
  \draw [-,thin,dotted] (d6_1221) -- (d7_1231);
  \draw [-,thin,dotted] (d6_1221) -- (d7_1321);
  \draw [-,thin,dotted] (d6_123) -- (d7_1231);
  \draw [-,thin,dotted] (d6_123) -- (d7_133);
  \draw [-,thin,dotted] (d6_1311) -- (d7_13111);
  \draw [-,thin,dotted] (d6_1311) -- (d7_1321);
  \draw [-,thin,dotted] (d6_1311) -- (d7_1411);
  \draw [-,thin,dotted] (d6_132) -- (d7_1321);
  \draw [-,thin,dotted] (d6_132) -- (d7_133);
  \draw [-,thin,dotted] (d6_132) -- (d7_142);
  \draw [-,thin,dotted] (d6_141) -- (d7_1411);
  \draw [-,thin,dotted] (d6_141) -- (d7_142);
  \draw [-,thin,dotted] (d6_141) -- (d7_151);
  \draw [-,thin,dotted] (d6_15) -- (d7_151);
  \draw [-,thin,dotted] (d6_15) -- (d7_16);
  
  \draw [-,thin,dotted] (d5_122) -- (d6_123);
  \draw [-,thin,dotted] (d5_131) -- (d6_1311);
  \draw [-,thin,dotted] (d5_1211) -- (d6_1311);
  \draw [-,thin,dotted] (d3_111) -- (d4_121);
  \draw [-,thin,dotted] (d5_131) -- (d6_141);
  \draw [-,thin,dotted] (d2_11) -- (d3_111);
  \draw [-,thin,dotted] (d4_13) -- (d5_131);
  \draw [-,thin,dotted] (d5_131) -- (d6_132);
  \draw [-,thin,dotted] (d5_11111) -- (d6_111111);
  \draw [-,thin,dotted] (d1) -- (d2_11);
  \draw [-,thin,dotted] (d5_11111) -- (d6_12111);
  \draw [-,thin,dotted] (d5_122) -- (d6_1221);
  \draw [-,thin,dotted] (d4_121) -- (d5_122);
  \draw [-,thin,dotted] (d3_111) -- (d4_1111);
  \draw [-,thin,dotted] (d3_12) -- (d4_13);
  \draw [-,thin,dotted] (d5_14) -- (d6_15);
  \draw [-,thin,dotted] (d5_1211) -- (d6_1221);
  \draw [-,thin,dotted] (d4_1111) -- (d5_11111);
  \draw [-,thin,dotted] (d3_12) -- (d4_121);
  \draw [-,thin,dotted] (d5_14) -- (d6_141);
  \draw [-,thin,dotted] (d4_13) -- (d5_14);
  \draw [-,thin,dotted] (d4_121) -- (d5_1211);
  \draw [-,thin,dotted] (d5_122) -- (d6_132);
  \draw [-,thin,dotted] (d2_11) -- (d3_12);
  \draw [-,thin,dotted] (d4_121) -- (d5_131);
  \draw [-,thin,dotted] (d4_1111) -- (d5_1211);
  \draw [-,thin,dotted] (d5_1211) -- (d6_12111);
\end{tikzpicture}
\caption{\label{fig:OseqGraph3} The order relations among directly comparable elements of $\mathcal{G}_d,\ d = 1,\ldots,7$.}
\end{center}
\end{figure}
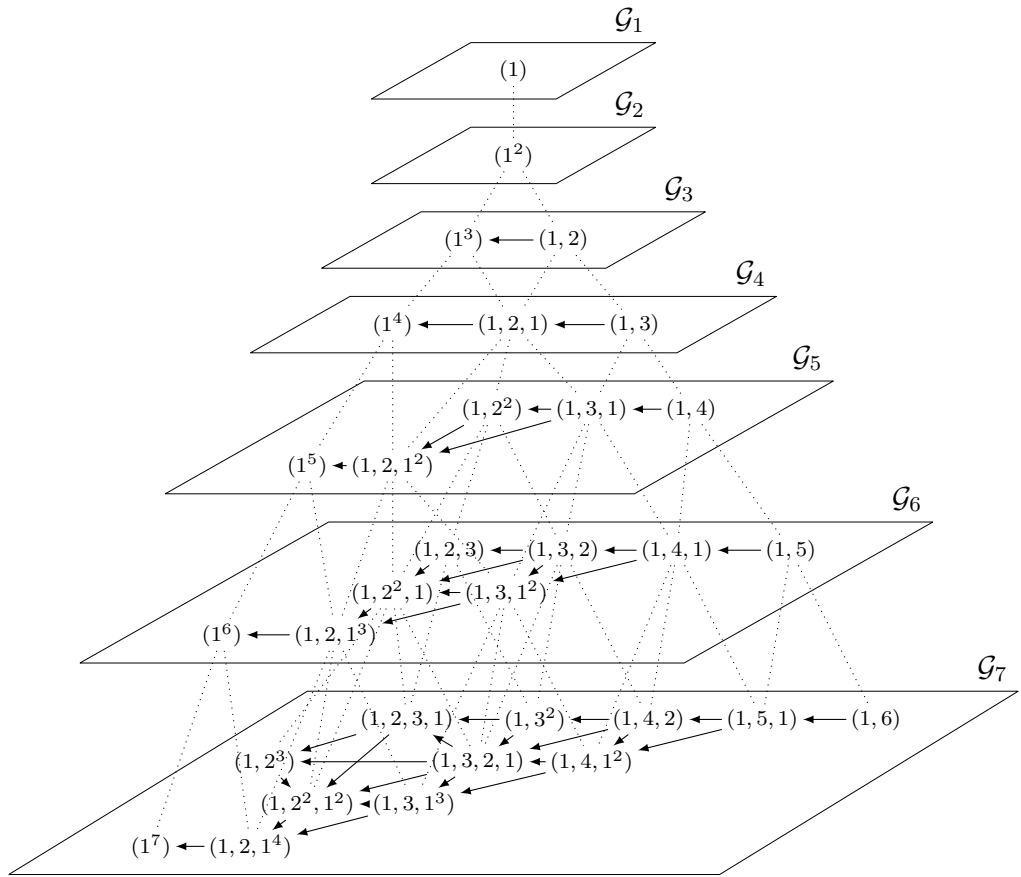

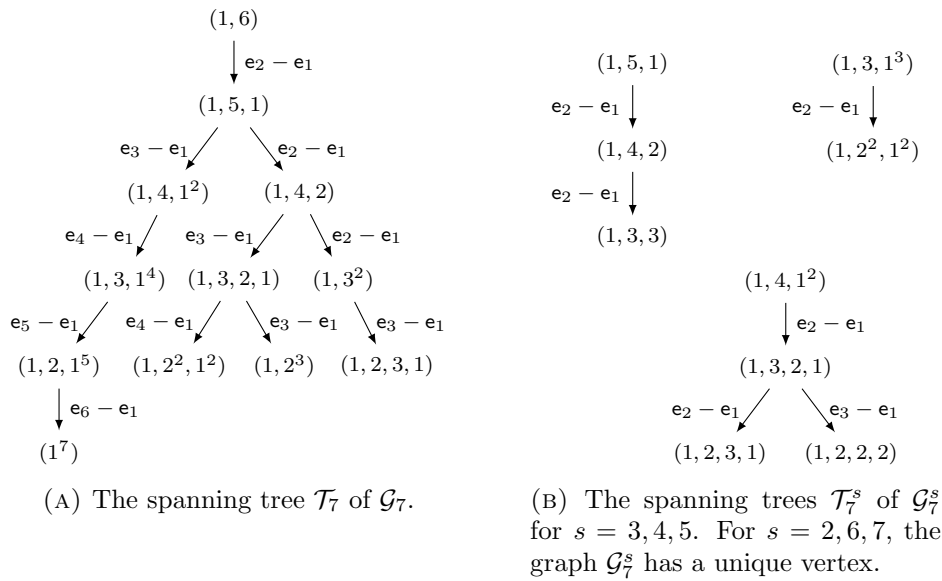
\begin{figure}[!ht]
\begin{center}
\subfloat[][The spanning tree $\mathcal{T}_7$ of $\mathcal{G}_7$.]
{\label{fig:OseqGraph4}
\begin{tikzpicture}[>=latex,scale=1.15]

\node (1111111) at (0,0) [] {\tiny $(1^7)$};

\node (121111) at (0,1) [] {\tiny $(1,2,1^5)$};
\node (12211) at (1.375,1) [] {\tiny $(1,2^2,1^2)$};
\node (1222) at (2.575,1) [] {\tiny $(1,2^3)$};
\node (1231) at (3.75,1) [] {\tiny $(1,2,3,1)$};

\draw [<-,thin] (1111111) --node[right]{\tiny $\mathsf{e}_6-\mathsf{e}_1$} (121111);

\node (13111) at (0.75,2) [] {\tiny $(1,3,1^4)$};
\node (1321) at (2,2) [] {\tiny $(1,3,2,1)$};
\node (133) at (3.25,2) [] {\tiny $(1,3^2)$};

\draw [<-,thin] (121111) --node[left]{\tiny $\mathsf{e}_5-\mathsf{e}_1$} (13111);
\draw [<-,thin] (12211) --node[left]{\tiny $\mathsf{e}_4-\mathsf{e}_1$} (1321);
\draw [<-,thin] (1222) --node[right]{\tiny $\mathsf{e}_3-\mathsf{e}_1$} (1321);
\draw [<-,thin] (1231) --node[right]{\tiny $\mathsf{e}_3-\mathsf{e}_1$} (133);

\node (1411) at (1.25,3) [] {\tiny $(1,4,1^2)$};
\node (142) at (2.75,3) [] {\tiny $(1,4,2)$};

\draw [<-,thin] (13111) --node[left]{\tiny $\mathsf{e}_4-\mathsf{e}_1$} (1411);
\draw [<-,thin] (1321) --node[left]{\tiny $\mathsf{e}_3-\mathsf{e}_1$} (142);
\draw [<-,thin] (133) --node[right]{\tiny $\mathsf{e}_2-\mathsf{e}_1$} (142);

\node (151) at (2,4) [] {\tiny $(1,5,1)$};

\draw [<-,thin] (1411) --node[left]{\tiny $\mathsf{e}_3-\mathsf{e}_1$} (151);
\draw [<-,thin] (142) --node[right]{\tiny $\mathsf{e}_2-\mathsf{e}_1$} (151);

\node (16) at (2,5) [] {\tiny $(1,6)$};
\draw [<-,thin] (151) --node[right]{\tiny $\mathsf{e}_2-\mathsf{e}_1$} (16);

\end{tikzpicture}
}
\qquad
\subfloat[][The spanning trees $\mathcal{T}_7^s$ of $\mathcal{G}_7^s$ for $s=3,4,5$. For $s=2,6,7$, the graph $\mathcal{G}_7^s$ has a unique vertex.]
{\label{fig:OseqGraph5}
\begin{tikzpicture}[>=latex,scale=1.15]

\node (s151) at (6.25,5) [] {\tiny $(1,5,1)$};
\node (s142) at (6.25,4) [] {\tiny $(1,4,2)$};
\node (s133) at (6.25,3) [] {\tiny $(1,3,3)$};
\draw [->,thin] (s151) --node[left]{\tiny $\mathsf{e}_2-\mathsf{e}_1$} (s142);
\draw [->,thin] (s142) --node[left]{\tiny $\mathsf{e}_2-\mathsf{e}_1$} (s133);

\node (s1411) at (8,2.5) [] {\tiny $(1,4,1^2)$};
\node (s1321) at (8,1.5) [] {\tiny $(1,3,2,1)$};
\node (s1231) at (7.25,0.5) [] {\tiny $(1,2,3,1)$};
\node (s1222) at (8.75,0.5) [] {\tiny $(1,2,2,2)$};
\draw [->,thin] (s1411) --node[right]{\tiny $\mathsf{e}_2-\mathsf{e}_1$} (s1321);
\draw [->,thin] (s1321) --node[left]{\tiny $\mathsf{e}_2-\mathsf{e}_1$} (s1231);
\draw [->,thin] (s1321) --node[right]{\tiny $\mathsf{e}_3-\mathsf{e}_1$} (s1222);

\node (s13111) at (9,5) [] {\tiny $(1,3,1^3)$};
\node (s12211) at (9,4) [] {\tiny $(1,2^2,1^2)$};
\draw [->,thin] (s13111) --node[left]{\tiny $\mathsf{e}_2-\mathsf{e}_1$} (s12211);
\end{tikzpicture}
}
\caption{\label{fig:fixedMultiplicity} Graph descriptions of O-sequences with given multiplicity and length.}
\end{center}
\end{figure}

Now, we can state the strategy of a general algorithm for searching aCM genera. We choose the set of O-sequences corresponding to the considered constraints on multiplicity and length and, more precisely, the associated spanning tree $\widetilde{\mathcal{T}}$. Then, we perform a depth-first search on the tree using a LIFO (Last In First Out) procedure of visit of the vertices. Assume that, at some moment in the search, we stored in a list (resp.~a stack) the vertices whose existence we know, having visited their parents, but that we have not yet visited. We visit the first vertex $\mathsf{h}$ in the list (resp.~the top of the stack). There are three possible alternative actions:
\begin{enumerate}
\item[\textsc{a}.] if $g(\mathsf{h})$ is equal to the genus we are looking for, then we end the visit returning the O-sequence $\mathsf{h}$;
\item[\textsc{b}.] if $g(\mathsf{h})$ is greater than the genus we are looking for, then we can avoid to visit the tree of descendants of $\mathsf{h}$, as the genus increases along the edges (Proposition \ref{whyOseqGraph} and Corollary \ref{cor:genusOrder});
\item[\textsc{c}.] if $g(\mathsf{h})$ is smaller than the genus we are looking for, then we need to visit the tree of descendants of $\mathsf{h}$, so we add the children of $\mathsf{h}$ in the tree $\widetilde{\mathcal{T}}$ at the beginning of the list (resp.~at the top of the stack) containing the vertices still to be visited.
\end{enumerate}

\begin{algorithm}[!ht]
\caption{\label{alg:genusTreeSearch} The algorithm for searching aCM genera with given constraints on the multiplicity and the length of the O-sequences. A trial version of this algorithm is available at \href{http://www.paololella.it/HSC/Finite_O-sequences_and_ACM_genus.html}{\small \tt http://www.paololella.it/HSC/Finite\_O-sequences\_and\_ACM\_genus.html}}
\begin{algorithmic}[1]
\Procedure{genusSearch}{$g,\widetilde{\mathcal{T}}$}
\Require \begin{itemize} \item[] $g$, a non-negative integer. \item[]
 \hspace{0.65cm}$\widetilde{\mathcal{T}}$, a spanning tree chosen among $\mathcal{T}$, $\mathcal{T}_d$, $\mathcal{T}^s$ and $\mathcal{T}_d^s$.\end{itemize}
\Ensure an O-sequence $\mathsf{h}$ such that $g(\mathsf{h}) = g$ (if it exists).
\State $\textsf{stack} := \{\textsc{root}(\widetilde{\mathcal{T}})\}$;
\While{$\textsf{stack} \neq \emptyset$}
\State $\mathsf{h} := \textsc{removeFirst}(\textsf{stack})$;
\If{$g(\mathsf{h}) = g$}
\Return $\mathsf{h}$;
\ElsIf{$g(\mathsf{h}) < g$}
\State $\textsc{addFirst}(\textsf{stack},\textsc{children}(\mathsf{h},\widetilde{\mathcal{T}}))$;
\EndIf
\EndWhile
\EndProcedure
\end{algorithmic}
\end{algorithm}


\section{Combinatorial ranges}
\label{sec:combRanges}

From now on, we assume $d>2$, as $\mathcal{G}_d$ has only one element for $d\in\{1,2\}$.

For convenience, we let $G_d$ (resp.~$G_d^s$) be the set of all the arithmetic genera of the aCM curves of degree $d$ (resp.~of degree $d$ with $h$-vector of length $s$), i.e.~$G_d:=\{ g({\mathsf{h}})\ \vert\ {\mathsf{h}} \in \mathcal{G}_d\}$ (resp.~$G_d^s:=\{ g({\mathsf{h}})\ \vert\ {\mathsf{h}} \in \mathcal{G}_d^s\}$). 

Looking at the graph $\mathcal{G}_d$, we immediately can observe the well known fact that $G_d\subseteq \big\{0,\ldots,\binom{d-1}{2}\big\}$ (see \cite[Theorem 3.1]{H94}). Denoting by $[a,b]$ the set of integers $\{n \in \mathbb{N}\ \vert\ a \leq n \leq b\}$, we let $R_d:=\big[0,\binom{d-1}{2}\big]$. In the range $R_d$ we single out smaller suitable ranges, taking into account also the length of the O-sequences.

Recall that, by the partial order $\prec$ introduced in Definition \ref{ordering} and by Corollary \ref{cor:genusOrder}, we have $\min({G}_d^s)= g(\min (\mathcal{G}_d^s))=\mathsf{g}^s=\binom{s-1}{2}$, thus $\mathsf{g}^s<\mathsf{g}^{s+1}$ and $\mathsf{g}^{s+1}-\mathsf{g}^s=s-1$. In order to obtain an analogous result about a maximum, we extend the partial order $\prec$ to the following total order on $\mathcal G_d^s$.

\begin{definition}\label{total order}
Given two O-sequences ${\mathsf h}=(1,h_1,\ldots,h_{s-1})$ and ${\mathsf h}'=(1,h'_1,\ldots,h'_{s-1})$ of $\mathcal G_d^s$, we denote by $<$ the total order on $\mathcal G_d^s$ such that ${\mathsf h} < {\mathsf h}'$ if $h_\ell<h'_\ell$, where $\ell:=\max\{j : h_j\not= h'_j\}$. 
\end{definition}

Although the usual order on $\mathbb N$ does not induce on ${\mathcal G}_d^s$ the total order $<$ (see  Example \ref{order not induced}), we notice that $\min_{\prec}({\mathcal G}_d^s)= \min({\mathcal G}_d^s)$ with respect to $<$. Furthermore, we can consider also $\max({\mathcal G}_d^s)$ with respect to $<$ and obtain the following non obvious result.

\begin{theorem}\label{th:degrevlex}
Let ${\mathsf h}=(1,h_1,\ldots,h_{s-1})$ and ${\mathsf k}=(1,k_1,\ldots,k_{s-1})$ be two O-sequences of $\mathcal G_d^s$. If ${\mathsf k} < {\mathsf h}$ and $g({\mathsf k}) > g({\mathsf h})$, then there is an O-sequence $\bar{\mathsf h} \in \mathcal G_d^s$ such that $\bar{\mathsf h} > {\mathsf h}$  and $g(\bar{\mathsf h}) > g({\mathsf k})$. Thus, $\max(G_d^s) = g (\max (\mathcal{G}_d^s))$.
\end{theorem}

\begin{proof}
We can assume $s-1=\max\{j : h_j\not= k_j\}$, hence $h_{s-1}>k_{s-1}$ because  ${\mathsf h} > {\mathsf k}$. By the hypotheses, we have
\[
g({\mathsf h})=\sum_j^{s-2}(j-1)h_j +(s-2)h_{s-1} < \sum_j^{s-2} (j-1)k_j + (s-2)k_{s-1}=g({\mathsf k})
\]
which implies there exists the integer $t:=\max\{j\in\{2,\ldots,s-2\} : h_j<k_j\}$ and so
\begin{equation}\label{eq:degrevlex}
\begin{array}{llllllll}
(1, & h_1, &\ldots, & h_{t}, & h_{t+1}, &\ldots, & h_{s-2}, & h_{s-1}) \\
    &      &        &\wedge & \veebar&        &\veebar    &\vee        \\
(1, & k_1, &\ldots, & k_{t}, & k_{t+1}, &\ldots, & k_{s-2}, & k_{s-1}) 
\end{array}
\end{equation}  
that is 
\begin{equation*}
\begin{cases}
h_t < k_t,\\
h_i \geq k_i,& t+1 \leq i \leq s-2,\\
h_{s-1} > k_{s-1}.
\end{cases}
\end{equation*}
Note that $k_t^{\langle t \rangle}\geq h_t^{\langle t\rangle} \geq h_{t+1}\geq k_{t+1}$. Hence, we can consider the O-sequence ${\mathsf h}' := {\mathsf k} - b \mathsf{e}_t + \sum_{j=t+1}^{s-1} c_{j} \mathsf{e}_j$, where
\[
 b = \min \left\{k_t-h_t,\sum_{j=t+1}^{s-1} h_j - k_j\right\} \quad \text{and}\quad c_j = \min \left\{ h_j-k_j, b - \sum_{i=t+1}^{j-1} c_i\right\}
\]
 and $h'_j\leq h_j$ for every $j>t$.

The corresponding genus of ${\mathsf h}'$ is 
\[
g({\mathsf h}')=g({\mathsf k})- (t-1)b + \sum_{j=t+1}^{s-1} (j-1)c_j > g({\mathsf k}) > g({\mathsf h}).
\]
If needed, replacing the O-sequence ${\mathsf k}$ by ${\mathsf h}'$ and repeating the same argument as before, we obtain an O-sequence ${\mathsf h}'$ with $h'_j=h_j$ for every $j>t$ and $g({\mathsf h}')>g({\mathsf h})$. If ${\mathsf h}' < {\mathsf h}$, we can repeat the same argument as before until we obtain an O-sequence $\bar{\mathsf h}$ with $\bar h_j=h_j$ for every $j>t$ and $\bar h_t\geq h_t+1$.
\end{proof}

\begin{example}\label{order not induced}
(a) Consider the two O-sequences ${\mathsf h}=(1,6,4,2,1)$ and ${\mathsf k}=(1,4,7,1,1)$ of $\mathcal G_{14}^5$. We have  ${\mathsf h} > {\mathsf k}$ and $11=g({\mathsf h}) < g({\mathsf k})=12$ as in the hypotheses of Theorem \ref{th:degrevlex}. In this case, we obtain $t=2$, $b = \min\{3,1\} = 1$, $c_3 = \min \{1,3\} = 1$ and $c_4 = \min\{ 0,2\} = 0$, so that $\bar{\mathsf h}=\mathsf{k}-\mathsf{e}_2+\mathsf{e}_3 =(1,4,6,2,1)$ with genus $g(\bar{\mathsf h})=13 > g(\mathsf{k})$ and $\bar{\mathsf h}> {\mathsf h}$.

(b) Consider the two O-sequences ${\mathsf h}=(1,13,3,3,3)$ and ${\mathsf k}=(1,6,13,2,1)$ of $\mathcal G_{23}^5$. We have ${\mathsf h} > {\mathsf k}$ and $18 = g({\mathsf h}) < g({\mathsf k})=20$. Applying Theorem \ref{th:degrevlex}, as $t = 2$, $b = \min \{10, 3\} = 3$, $c_3 = \min\{1,10\} = 1$ and $c_4 = \min\{2,9\} = 2$, we determine $\bar{\mathsf h} = \mathsf{k} - 3\mathsf{e}_2 + \mathsf{e}_3 + 2\mathsf{e}_4 = (1,6,10,3,3) > \mathsf{h}$ and $g(\bar{\mathsf{h}}) = 18+2+3 = 21 > g(\mathsf{k})$.
\end{example}

Looking again at the graph $\mathcal{G}^s$, we can find a way to detect $g(\max({\mathcal G}_d^s))$. We first note that, if $d<s$, then $\mathcal G_d^s$ is empty and if $d=s$, then we have a unique O-sequence $(1^s)$ corresponding to a plane curve of degree $s$, i.e.~with genus $\binom{s-1}{2}$. For $d=s+1$ we have the unique O-sequence $(1,2,1^{s-2})$, obtained from $(1^{s})$ by increasing $h_1$ by $1$ and corresponding to a curve of degree $s+1$ and genus $\binom{s-1}{2}$. In the other cases, we deduce $\max({\mathcal G}_d^s)$ assuming to know the O-sequence  $\mathsf{h}=\max({\mathcal G}_{d-1}^s)$ and consequently the genus $g(\mathsf{h})=\sum_{j=2}^{s-1} h_j(j-1)=\max(G_{d-1}^s)$ (Theorem \ref{th:degrevlex}). Next result shows how to find $\max({\mathcal G}_d^s)$ and then $g(\max({\mathcal G}_d^s))$.

\begin{proposition}\label{dim procedura paolo}
Given any $d> s\geq 3$, let ${\mathsf h}=\max({\mathcal G}_{d-1}^s)$. If $\imath$ is the highest index such that ${\mathsf h}+{\mathsf e}_\imath$ is an O-sequence in $\mathcal G_d^s$, then $\max({\mathcal G}_{d}^s)={\mathsf h}+{\mathsf e}_\imath$ and $g(\max({\mathcal G}_d^s))=g(\max({\mathcal G}_{d-1}^s))+\imath-1$.
\end{proposition}

\begin{proof}
By the assumption, we have $h_\imath< h_{\imath-1}^{\langle \imath-1 \rangle}$, so that $h_\imath +1\leq h_{\imath-1}^{\langle \imath-1 \rangle}$ and $h_{\imath+r}= h_{\imath+r-1}^{\langle \imath+r-1 \rangle}$, for every $1\leq r\leq s-1-\imath$, that is:
$${\mathsf h} = (1,\ldots,h_\imath, h_\imath^{\langle \imath \rangle}, h_{\imath+1}^{\langle \imath+1 \rangle},\ldots,h_{s-2}^{\langle s-2 \rangle})$$ and
$${\mathsf h}+{\mathsf e}_\imath = (1,\ldots,h_\imath +1, h_\imath^{\langle \imath \rangle}, h_{\imath+1}^{\langle \imath+1 \rangle},\ldots,h_{s-2}^{\langle s-2 \rangle}).$$
For every ${\mathsf h}' \in \mathcal G_{d-1}^s\setminus \{{\mathsf h}\}$, consider the integer $\ell:=\max\{j: \, h_j\neq h'_j\}$. Then, we have $h'_\ell < h_\ell$ and $h'_{\ell+r} = h_{\ell+r}$, for every $1\leq r\leq s-1-\ell$, because ${\mathsf h}=\max({\mathcal G}_{d-1}^s)$. Note that we have $\ell<\imath$, otherwise $h'_\ell < h_\ell$ would imply $h'_{\ell+1}\leq {h'}_\ell^{\langle \ell \rangle} < h_\ell^{\langle \ell \rangle}=h_{\ell+1}$, against the definition of $\ell$. Therefore, 
$${\mathsf h}'= (1,\ldots,h'_\ell,\ldots, h_\imath, h_\imath^{\langle \imath \rangle}, \ldots,h_{s-2}^{\langle s-2 \rangle}).$$

If there were an O-sequence ${\mathsf h}' \in \mathcal G_{d-1}^s$ such that ${\mathsf h}'+{\mathsf e}_\lambda > {\mathsf h}+{\mathsf e}_\imath$ for some index $\lambda$ such that ${\mathsf h}'+{\mathsf e}_\lambda \in \mathcal G_d^s$, then $\imath < \lambda$.
We have seen that ${\mathsf h}$ and ${\mathsf h}'$ certainly have equal entries for indices greater than or equal to $\imath$ and $h_\imath+1>h_\imath=h'_\imath$. But, for indices $j>\imath$, the value $h'_j=h_j=h_{j-1}^{\langle j-1 \rangle}$ cannot be increased by the definition of O-sequences. Thus, we obtain $\max({\mathcal G}_{d}^s)={\mathsf h}+{\mathsf e}_\imath$.
The last assertion follows by Theorem \ref{th:degrevlex} and formula \eqref{genus}.
\end{proof}

For every $d > 2$ and $s\in \{\lfloor \frac{d}{2}\rfloor +1,\ldots,d\}$, we let 
\begin{equation}\label{eq:max}
{\mathsf h}_s(d):=(1,2^{d-s},1^{2s-d-1}) \quad \text{ and } \quad {\mathsf g}_s(d):=g({\mathsf h}_s(d))=\tbinom{s-1}{2}+\tbinom{d-s}{2}.
\end{equation}
Then, we have: \ $\max({\mathcal G}_d^s)={\mathsf h}_s(d)$, ${\mathsf g}^d=\binom{d-1}{2}={\mathsf g}_{d}(d)$ and ${\mathsf g}^{d-1}=\binom{d-2}{2}={\mathsf g}_{d-1}(d)$.

\begin{remark}
Another description of the maximal genus of a range $R_d^s$ could be set in terms of minimal Hilbert functions with a constant Hilbert polynomial and a given regularity (see \cite[Examples 4.6 and 4.8]{R} and \cite{CMreg}). By the way, the combinatorial description we provide here arises in a very natural way and gives more information, at least from a computational point of view.  
\end{remark}

The previous results together with those of Sections \ref{sec:combinatorial} suggest to consider the following smaller ranges in $R_d$.

\begin{definition}
For every $d\geq s\geq 2$, the set of integers between ${\mathsf g}^s$ and $\max(G_d^s)$ is called {\em $(d,s)$-range} and denoted by $R_d^s$, i.e.~$R_d^s:=\left\{\alpha\in\mathbb{N}\ \vert\ \tbinom{s-1}{2}={\mathsf g}^s\leq\alpha\leq \max(G_d^s) \right\}$.
\end{definition}

\begin{corollary}\label{cor:ranges}
For every $d\geq s \geq 2$, the arithmetic genus of an aCM curve of degree $d$ having $h$-vector of length $s$ belongs to the range $R^s_d$.
\end{corollary}

\begin{proof}
The statement follows by Corollary \ref{cor:genusOrder}, Theorem \ref{th:degrevlex} and Proposition \ref{dim procedura paolo}.
\end{proof}

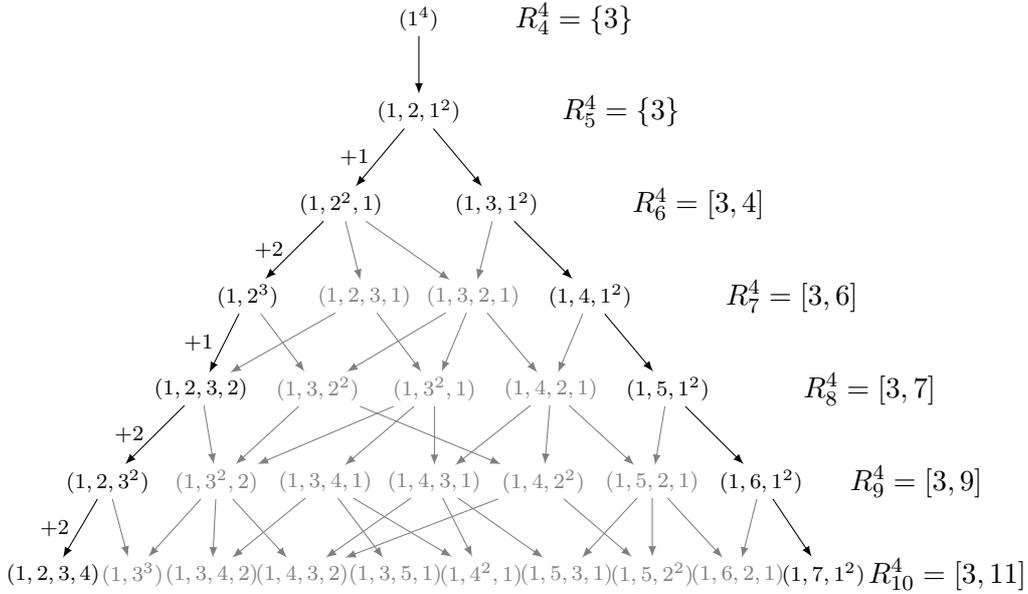
\begin{figure}[!ht]
\begin{center}
\begin{tikzpicture}[>=latex,line join=bevel,scale=0.58]
  \node (d4_1111) at (300bp,370bp) [inner sep=2pt] {\tiny $(1^4)$};
  
  \node at (400bp,370bp) [] {\small $R_4^4 = \{3\}$};
  
  \node (d5_1211) at (300bp,310bp) [inner sep=2pt] {\tiny $(1,2,1^2)$};
  
  \node at (430bp,310bp) [] {\small $R_5^4 = \{3\}$};

  \node (d6_1221) at (250bp,250bp) [inner sep=2pt] {\tiny $(1,2^2,1)$};
  \node (d6_1311) at (350bp,250bp) [inner sep=2pt] {\tiny $(1,3,1^2)$};
  
  \node at (480bp,250bp) [] {\small $R_6^4 = [3,4]$};

  \node (d7_1222) at (190bp,190bp) [inner sep=2pt] {\tiny $(1,2^3)$};
  \node (d7_1231) at (265bp,190bp) [inner sep=2pt,black!50] {\tiny $(1,2,3,1)$};
  \node (d7_1321) at (335bp,190bp) [inner sep=2pt,black!50] {\tiny $(1,3,2,1)$};
  \node (d7_1411) at (410bp,190bp) [inner sep=2pt] {\tiny $(1,4,1^2)$};
  
  \node at (540bp,190bp) [] {\small $R_7^4 = [3,6]$};
  
  \node (d8_1232) at (160bp,130bp) [inner sep=2pt] {\tiny $(1,2,3,2)$};
  \node (d8_1322) at (235bp,130bp) [inner sep=2pt,black!50] {\tiny $(1,3,2^2)$};
  \node (d8_1331) at (310bp,130bp) [inner sep=2pt,black!50] {\tiny $(1,3^2,1)$};
  \node (d8_1421) at (385bp,130bp) [inner sep=2pt,black!50] {\tiny $(1,4,2,1)$};
  \node (d8_1511) at (460bp,130bp) [inner sep=2pt] {\tiny $(1,5,1^2)$};

  \node at (590bp,130bp) [] {\small $R_8^4 = [3,7]$};
  
  \node (d9_1233) at (100bp,70bp) [inner sep=2pt] {\tiny $(1,2,3^2)$};
  \node (d9_1332) at (170bp,70bp) [inner sep=2pt,black!50] {\tiny $(1,3^2,2)$};
  \node (d9_1341) at (240bp,70bp) [inner sep=2pt,black!50] {\tiny $(1,3,4,1)$};
  \node (d9_1431) at (310bp,70bp) [inner sep=2pt,black!50] {\tiny $(1,4,3,1)$};
  \node (d9_1422) at (380bp,70bp) [inner sep=2pt,black!50] {\tiny $(1,4,2^2)$};
  \node (d9_1521) at (450bp,70bp) [inner sep=2pt,black!50] {\tiny $(1,5,2,1)$};
  \node (d9_1611) at (520bp,70bp) [inner sep=2pt] {\tiny $(1,6,1^2)$};
  
  \node at (620bp,70bp) [] {\small $R_9^4 = [3,9]$};

  \node (d10_1234) at (65bp,10bp) [inner sep=2pt] {\tiny $(1,2,3,4)$};
  \node (d10_1333) at (116bp,10bp) [inner sep=2pt,black!50] {\tiny $(1,3^3)$};
  \node (d10_1342) at (167bp,10bp) [inner sep=2pt,black!50] {\tiny $(1,3,4,2)$};
  \node (d10_1432) at (225bp,10bp) [inner sep=2pt,black!50] {\tiny $(1,4,3,2)$};
  \node (d10_1351) at (285bp,10bp) [inner sep=2pt,black!50] {\tiny $(1,3,5,1)$};
  \node (d10_1441) at (340bp,10bp) [inner sep=2pt,black!50] {\tiny $(1,4^2,1)$};
  \node (d10_1531) at (395bp,10bp) [inner sep=2pt,black!50] {\tiny $(1,5,3,1)$};
  \node (d10_1522) at (450bp,10bp) [inner sep=2pt,black!50] {\tiny $(1,5,2^2)$};
  \node (d10_1621) at (505bp,10bp) [inner sep=2pt,black!50] {\tiny $(1,6,2,1)$};
  \node (d10_1711) at (560bp,10bp) [inner sep=2pt] {\tiny $(1,7,1^2)$};
  
    \node at (640bp,10bp) [] {\small $R_{10}^4 = [3,11]$};

  \draw [->] (d4_1111) -- (d5_1211);
  \draw [->] (d5_1211) -- (d6_1311);
  \draw [->] (d5_1211) --node[left]{\tiny $+1$} (d6_1221);
  \draw [->] (d6_1221) --node[left]{\tiny $+2$} (d7_1222);
  \draw [->] (d7_1222) --node[left]{\tiny $+1$} (d8_1232);
  \draw [->] (d8_1232) --node[left]{\tiny $+2$} (d9_1233);
  \draw [->] (d9_1233) --node[left]{\tiny $+2$} (d10_1234);
  \draw [->] (d6_1311) -- (d7_1411);
  \draw [->] (d7_1411) -- (d8_1511);
  \draw [->] (d8_1511) -- (d9_1611);
  \draw [->] (d9_1611) -- (d10_1711);
  
  \draw [->,black!50] (d9_1431) -- (d10_1531);
  \draw [->,black!50] (d9_1431) -- (d10_1441);
  \draw [->,black!50] (d8_1232) -- (d9_1332);
  \draw [->,black!50] (d9_1611) -- (d10_1621);
  \draw [->,black!50] (d9_1332) -- (d10_1342);
  \draw [->,black!50] (d8_1322) -- (d9_1332);
  \draw [->,black!50] (d8_1331) -- (d9_1332);
  \draw [->,black!50] (d9_1341) -- (d10_1441);
  \draw [->,black!50] (d8_1331) -- (d9_1341);
  \draw [->,black!50] (d8_1331) -- (d9_1431);
  \draw [->,black!50] (d6_1221) -- (d7_1321);
  \draw [->,black!50] (d8_1421) -- (d9_1431);
  \draw [->,black!50] (d9_1422) -- (d10_1522);
  \draw [->,black!50] (d9_1521) -- (d10_1621);
  \draw [->,black!50] (d8_1511) -- (d9_1521);
  \draw [->,black!50] (d9_1431) -- (d10_1432);
  \draw [->,black!50] (d6_1221) -- (d7_1231);
  \draw [->,black!50] (d9_1341) -- (d10_1351);
  \draw [->,black!50] (d7_1222) -- (d8_1322);
  \draw [->,black!50] (d9_1332) -- (d10_1333);
  \draw [->,black!50] (d7_1321) -- (d8_1331);
  \draw [->,black!50] (d9_1422) -- (d10_1432);
  \draw [->,black!50] (d8_1322) -- (d9_1422);
  \draw [->,black!50] (d9_1521) -- (d10_1522);
  \draw [->,black!50] (d9_1233) -- (d10_1333);
  \draw [->,black!50] (d7_1321) -- (d8_1421);
  \draw [->,black!50] (d9_1521) -- (d10_1531);
  \draw [->,black!50] (d6_1311) -- (d7_1321);
  \draw [->,black!50] (d7_1231) -- (d8_1331);
  \draw [->,black!50] (d7_1321) -- (d8_1322);
  \draw [->,black!50] (d9_1332) -- (d10_1432);
  \draw [->,black!50] (d7_1411) -- (d8_1421);
  \draw [->,black!50] (d8_1421) -- (d9_1521);
  \draw [->,black!50] (d9_1341) -- (d10_1342);
  \draw [->,black!50] (d8_1421) -- (d9_1422);
  \draw [->,black!50] (d7_1231) -- (d8_1232);
\end{tikzpicture}
\caption{\label{fig:ranges} The ranges $R_d^4$ for $d=4,\ldots,10$. In the picture, the edges on the left are labeled with the corresponding increase of the genus.}
\end{center}
\end{figure}


\section{\texorpdfstring{Unattainable aCM genera in $R_d$}{Unattainable aCM genera in R\_d}}\label{sec:gaps}

Recall that we are denoting by $R_d$ the range $\big[0,\binom{d-1}{2}\big]$ and that $G_d\subseteq R_d$.

\begin{definition}\label{def:gap}
An integer in $R_d \setminus G_d$ is called a {\em gap in $R_d$}.
\end{definition}

\begin{example}\label{ex:gaps ovvi}
The integers in the range $\big[\binom{d-2}{2}+1,\binom{d-1}{2}-1\big]$ are gaps in $R_d$. More generally, every integer of $R_d$ not contained in any $(d,s)$-range is a gap.
\end{example}

Next result allows us to characterize the consecutive $(d,s)$-ranges that are \emph{separated}, i.e.~ranges $R_d^s$ and $R_d^{s+1}$ such that $\mathsf{g}^{s+1} - \max(G_d^s) > 1$. 

\begin{proposition}\label{easy gaps}
For any $d > 2$, we have
\[
\max(G_d^s) < \mathsf{g}^{s+1}-1 \quad \Longleftrightarrow\quad \frac{2d+1-\sqrt{8d-15}}{2}< s \leq d-1.
\]
Thus, the integers in $[\max(G_d^s)+1,{\mathsf g}^{s+1}-1]$ are gaps in $R_d$, for $\frac{2d+1-\sqrt{8d-15}}{2}< s \leq d-1$.
\end{proposition}

\begin{proof}
For $s \geq \left\lfloor\frac{d}{2}\right\rfloor + 1$, by \eqref{eq:max} we have:  
\[
\mathsf{g}_s(d) < \mathsf{g}^{s+1}-1 \quad \Longleftrightarrow\quad \tbinom{s-1}{2}+\tbinom{d-s}{2} < \tbinom{s}{2}-1.
\]
Hence 
\[
 {\mathsf g}_{s}(d) - {\mathsf g}^{s+1} + 1 = \tfrac{s^2-(2d+1)s+d^2-d+4}{2} < 0\ \Rightarrow\ \tfrac{2d+1-\sqrt{8d-15}}{2} < s < \tfrac{2d+1+\sqrt{8d-15}}{2},
 \]
and thus $\mathsf{g}_s(d) < \mathsf{g}^{s+1}-1$ if and only if $\frac{2d+1-\sqrt{8d-15}}{2}< s \leq d-1$, because $\frac{2d+1-\sqrt{8d-15}}{2} >\left\lfloor\frac{d}{2}\right\rfloor $, $\tfrac{2d+1+\sqrt{8d-15}}{2} > d-1$ and $\frac{2d+1-\sqrt{8d-15}}{2} > d-1$ implies $d < 3$.

To prove that there are no other pairs of separated ranges, we notice that $\mathsf{g}_s(d) \geq \mathsf{g}^{s+1}-1$ implies $\mathsf{g}_{s-1}(d) \geq \mathsf{g}^{s}-1$, for every $s$. Indeed, as $\mathsf{g}^{s} = \mathsf{g}^{s+1}-(s-1)$ and $\mathsf{g}_{s}(d) \leq \mathsf{g}_{s-1}(d) + (s-2)$ by Proposition \ref{dim procedura paolo}, we have
\[
\mathsf{g}_{s-1}(d) - \mathsf{g}^{s}+1 \geq \mathsf{g}_{s}(d) - (s-2) - \mathsf{g}^{s+1}+ (s-1) + 1 > \mathsf{g}_s(d) -\mathsf{g}^{s+1}+1 \geq 0. \qedhere
\]
\end{proof}

\begin{example}\label{ex:d=12}
For every $d\leq 11$, the gaps in $R_d$ are only those described in Proposition \ref{easy gaps}. For $d=12$, in addition to the gaps described in Proposition \ref{easy gaps}, we find by direct computation a unique further gap $\bar g=26$, belonging only to the range $R^8_{12}=[21,28]$.
\end{example}

\begin{example}\label{ex:d=15}
By a direct computation of the finite admissible O-sequences, we note that for $d=15$ the integer $\bar g=25$ belongs to the ranges $R^{6}_{15}$ and $R^{5}_{15}$. Nevertheless, whereas for each $\mathsf{h} \in R_{15}^5$ we have $g(\mathsf{h}) \neq 25$, there is $\overline{\mathsf{h}} = (1,3,3,4,2,2) \in R_{15}^6$ such that $g(\overline{\mathsf{h}}) = 25$.  
\end{example}

Example \ref{ex:d=15} suggests the following definition.

\begin{definition}
An integer in the range $R_d^s$ is called a {\em hole} of the range $R_d^s$ if it is not the arithmetic genus of an aCM curve $C$ of degree $d$ with $h$-vector of length $s$.
\end{definition}

\begin{remark}
Not every hole is a gap. For instance, Example \ref{ex:d=15} tells us that the integer $25$ is not a gap in $R_{15}$, although it is a hole of $R^{5}_{15}$. While Example \ref{ex:d=12} attests that the hole $26$  of $R^8_{12}$  is actually a gap in $R_{12}$.
\end{remark}

Notice that for $s = d-1,d-2,d-3$ there are no holes in $R_d^s$.
Now, we detect some values of $d$ and $s$ for which in the ranges $R_d^s$ there exist certain special gaps and we point out some particular holes which are also gaps, belonging to {\em parts} of different $(d,s)$-ranges not overlapping each other.

\begin{lemma}\label{lemmaHoles}
For every $d$ and $s$ such that $7 \leq \left\lfloor\frac{d}{2}\right\rfloor+1 \leq s \leq d-4$, the integers $\mathsf{g}_s(d)-(d-s-3),\ldots,\mathsf{g}_s(d)-1$ are holes in the range $R_d^s$.
\end{lemma}

\begin{proof}
As we saw in \eqref{eq:max}, the maximal genus $\mathsf{g}_s(d)$ in $R_d^s$ arises from the O-sequence $\mathsf{h}_s(d) = (1,2^{d-s},1^{2s-d-1})$. In the graph $\mathcal{G}_d^s$, the only edges involving this vertex are $\mathsf{e}_{d-2}-\mathsf{e}_1$ and $\mathsf{e}_{d-s}-\mathsf{e}_2$. Hence, by Corollary \ref{cor:genusOrder}, for each $\mathsf{h} \in \mathcal{G}_d^s\setminus \{\mathsf{h}_s(d)\}$
\[
\begin{split}
g(\mathsf{h}) &{}\leq \max \left\{ g\big(\mathsf{h}_s(d) - (\mathsf{e}_{d-s}-\mathsf{e}_1)\big) , g\big(\mathsf{h}_s(d) -(\mathsf{e}_{d-s}- \mathsf{e}_2)\big)\right\}\\
& {} = \max\left\{ \mathsf{g}_s(d) - (d-s-1), \mathsf{g}_s(d) - (d-s-2)\right\} = \mathsf{g}_s(d) - (d - s - 2). \qedhere
\end{split}
\]
\end{proof}

All the holes described in the previous lemma are surely gaps if we consider $s > \frac{2d+1-\sqrt{8d-15}}{2}$ as in Proposition \ref{easy gaps}. Indeed, is such cases these holes do not belong to any other range.

\begin{proposition}\label{2}
In the hypotheses of Lemma \ref{lemmaHoles}, for every $i=1,\ldots,d-s-3$, the hole $\mathsf{g}_s(d) - i$  is a gap if $s-1-\binom{d-s}{2}+i > 0$. More precisely,
\begin{enumerate}[(i)]
\item\label{it:2_i} the highest hole $\mathsf{g}_{d-4}(d) - 1 = \frac{d(d-11)}{2}+20$ is always a gap;
\item\label{it:2_ii} every hole described in Lemma \ref{lemmaHoles} is a gap if $s > \frac{2d-1-\sqrt{8d-31}}{2}$.
\end{enumerate}
\end{proposition}
\begin{proof}
The hole $\mathsf{g}_s(d) - i$ is a gap if $\mathsf{g}_s(d) - i < \mathsf{g}^{s+1}$, i.e.~$\binom{s}{2} - \binom{s-1}{2} - \binom{d-s}{2}+i = s-1 - \binom{d-s}{2}+i > 0$.
The proof of \emph{(\ref{it:2_i})} and \emph{(\ref{it:2_ii})} is a direct computation.
\end{proof}

\begin{example}
By Proposition \ref{2}, we find the following gaps in $R_{28}$: the gap $258$ belonging only to the range $R_d^{24}$, $240$ and $239$ belonging only to $R_d^{23}$, $224$, $223$ and $222$ belonging only to $R_d^{22}$ and $207$, $208$ and $209$ belonging to $R^{21}_{28}$. Anyway, by a direct computation we find also the gap $188$, actually the minimal one in $R_{28}$. 
\end{example}


\section{\texorpdfstring{Computation of the aCM genera for curves of degree $d$}{Computation of the aCM genera for curves of degree d}}
\label{sec:computation}

Proposition \ref{2} gives a characterization of the gaps in $R_d$ belonging to the {\em last part} of a $(d,s)$-range. We did not find analogous conditions for gaps belonging to the {\em first part} of a $(d,s)$-range. In particular, it seems hard to give a characterization of the minimal gap. Hence, we will look for an algorithmic method to recognize the gaps in $R_d$, avoiding to construct all the finite O-sequences of multiplicity $d$ thanks to a sort of {\em continuity} in the generation of the arithmetic genera. Denote by $G_d+a$ the set of all arithmetic genera of the aCM curves of degree $d$ augmented by a non-negative integer $a$.

\begin{lemma}\label{gradi successivi}
$G_d \supseteq \bigcup\limits_{j=1}^{d-1} \left(G_{j}+\tbinom{d-j}{2}\right)$.
\end{lemma}

\begin{proof}
Let $(1,h_1,\ldots,h_{s-1})$ be an O-sequence of multiplicity $j<d$ corresponding to an aCM genus $g$. Assuming $h_{i}^{\langle i \rangle} > h_{i+1}$, for some $i\in\{1,\ldots,s-2\}$, we can consider the finite O-sequence $(1,h_1,\ldots,h_{i+1}+1,\ldots,h_{s-1})$ of multiplicity $j+1$, corresponding to the genus $g+i$. Then, we can take also the finite O-sequence $(1,h_1,\ldots,h_{i+1}+1, h_{i+2}+1,\ldots,h_{s-1})$ of multiplicity $j+2$, corresponding to the genus $g+i+(i+1)$, and so on. Performing this construction from $i=1$ until $d-j$, we reach the desired conclusion. 
\end{proof}

\begin{remark}\label{rem:gradi successivi}
By the proof of Lemma \ref{gradi successivi}, we can observe that the arithmetic genera determined by the O-sequences $(1,h_1,\ldots,h_{s-1})$ with $h_i\geq h_{i+1}$, for every $0<i<s-1$, are included in those detected by Lemma \ref{gradi successivi}. For example, we have:
\footnotesize
\[
\begin{split}
G_1&{}=G_2 = \{0\}, \quad G_3= G_2\cup (G_1+1)=\{0,1\},\\
G_4&{}=  G_3\cup (G_2+1) \cup (G_1+3)=\{0,1,3\},\\
G_5&{}=  G_4\cup (G_3+1) \cup (G_2+3) \cup (G_1+6)=\{0,1,2,3,6\},\\
G_6&{}=  G_5\cup (G_4+1) \cup (G_3+3) \cup (G_2+6) \cup (G_1+10)=
\{0,1,2,3,4,6,10\},\\
G_7&{}\supset G_6\cup (G_5+1) \cup (G_4+3) \cup (G_3+6) \cup (G_2+10) \cup (G_1+15)=
\{0,1,2,3,4,6,7,10,15\}.
\end{split}
\]
\normalsize
Note that for the multiplicity $d=7$, we lose the arithmetic genus $g=5$ which corresponds to the finite O-sequence $(1,2,3,1)$. 
\end{remark}

Now, we exploit Lemma \ref{gradi successivi} obtaining large sets of aCM genera. To this aim, we define an increasing sequence $\{m_d\}_{d\geq 1}$ by the following procedure:
\begin{algorithmic}
\If{$ d = 1$}
\State $m_1:=0$;
\Else
\State $M := m_{d-1}$;
\For{$k = 2,\ldots,d-1$}
\If{$\binom{k}{2}-1 \leq M$}
\State $M = \max\{M,m_{d-k} + \binom{k}{2}\}$;
\EndIf
\EndFor
\State $m_d := M$; 
\EndIf
\end{algorithmic}

\begin{example}\label{ex:compare}
In the following table, we list the values of the sequence $\{m_d\}_{d\geq 1}$ and compare them with the values of ${\mathsf g}^{\lceil\frac{d}{2}\rceil +2}$, for $1\leq d \leq 45$:
\begin{table}[!ht]
\begin{center}
{\tiny
\begin{tabular}{r| ccccccccccccccc }
$d$ & $1$ & $2$ & $3$ & $4$ & $5$ & $6$ & $7$ & $8$ & $9$ & $10$ &$11$ & $12$ & $13$ & $14$ & $15$ \\
\hline
$m_d$  & $0$ & $0$ & $1$ &$1$ & $3$& $4$ & $4$ & $7$ & $11$ & $13$ & $18$  & $19$ & $19$ & $25$ & $32$ \\
\hline
${\mathsf g}^{\lceil\frac{d}{2}\rceil +2}$ & $1$ & $1$ & $3$ & $3$ & $6$ & $6$ & $10$ & $10$ & $15$ & $15$ & $21$ & $21$ & $28$ & $28$ & $36$ \\
\multicolumn{1}{c}{} & &&&&&&&&&&&&&& \\
$d$ & $16$ & $17$ & $18$ & $19$ & $20$ & $21$ & $22$ & $23$ & $24$ & $25$ & $26$ & $27$ & $28$ & $29$ & $30$ \\
\hline
$m_d$ & $40$ & $43$ &$52$ & $62$ &$73$ & $85$& $89$ & $102$ & $116$ & $118$ & $133$ & $149$  & $166$ & $184$ & $203$ \\
\hline
${\mathsf g}^{\lceil\frac{d}{2}\rceil +2}$ & $36$ & $45$ & $45$ & $55$ & $55$ & $66$ & $66$ & $78$ & $78$ & $91$ & $91$ & $105$ & $105$ & $120$ & $120$  \\
\multicolumn{1}{c}{} & &&&&&&&&&&&&&& \\
$d$ & $31$ & $32$ & $33$ & $34$ & $35$ & $36$ & 37 & 38 & 39 & 40 & 41 & 42 & 43 & 44 & 45  \\
\hline
$m_d$ & $208$ & $228$ & $229$ & $229$ & $250$ & $272$ & 295 & 319 & 344 & 370 & 376 & 403 & 431 & 460 & 490 \\
\hline
${\mathsf g}^{\lceil\frac{d}{2}\rceil +2}$ & $136$ & $136$ & $153$ & $153$ & $171$ & $171$ & 190 & 190 & 210 & 210 & 231 & 231 & 253 & 253 & 271 \\
\end{tabular}}
\end{center}
\end{table}
\end{example}

\begin{theorem}[Continuity]\label{global continuity} 
For all $d\geq 1$, every integer in $\{0,\ldots,m_d\}$ is the arithmetic genus of an aCM curve of degree $d$, i.e.~$\{0,\ldots,m_d\}\subseteq G_d$, and $m_d \geq {\mathsf g}^{\lceil\frac{d}{2}\rceil +2}$, for every $d\geq 18$.
\end{theorem}

\begin{proof} The first statement holds by Lemma \ref{gradi successivi} and by the definition of $m_d$. For the second affirmation, note that it is enough to consider odd degrees $d$. For $18\leq d\leq 36$, see the tables of Example \ref{ex:compare}. If $d\geq 37$, let $s:=\lceil\frac{d}{2}\rceil +2$. By construction and by induction, we know that $m_d \geq m_{d-1} \geq {\mathsf g}^{\lceil\frac{d-1}{2}\rceil +2}=\binom{s-2}{2}$. Hence, by the definition of $m_d$ we get
\[
m_d\geq \max\left\{m_{d-1},m_{d-(s-2)}+\tbinom{s-2}{2}\right\}.
\]
Being $d$ odd, we have $d-(s-2)=d-\lceil\frac{d}{2}\rceil= \lceil\frac{d}{2}\rceil -1 = s-3 \geq 18$. Thus, by induction we obtain $m_d\geq \binom{\lceil\frac{s-3}{2}\rceil+1}{2}+\binom{s-2}{2}$, because $m_{d-(s-2)}=m_{\lceil\frac{d}{2}\rceil-1}=m_{s-3}\geq {\mathsf g}^{\lceil\frac{s-3}{2}\rceil+2}$.

Note that $\binom{\lceil\frac{s-3}{2}\rceil+1}{2}+\binom{s-2}{2} \geq \binom{s-1}{2}$ if $\binom{\lceil\frac{s-3}{2}\rceil+1}{2}\geq s-2$, that is true for every $s\geq 10$.
\end{proof}

Theorem \ref{global continuity} gives a lower bound for the value assumed by $m_d$, for every $d\geq 18$. Anyway, we can obtain more information by a full application of Lemma \ref{gradi successivi} which, together with the algorithm $\textsc{genusSearch}$ (see Algorithm \ref{alg:genusTreeSearch}), provides an algorithm to compute all the arithmetic genera of the aCM curves of degree $d$, avoiding to construct all the finite O-sequences. The strategy consists of the following steps:
\begin{description}
\item[Step 1]\label{it:alg2step1} by Lemma \ref{gradi successivi}, we determine recursively the set of integers $\widetilde{G}_d \subset R_d$ that are certainly aCM genera . Let $\widetilde{G}_1 = \{0\}$, we have $\widetilde{G}_d = \bigcup_i \widetilde{G}_{i} + \binom{d-i}{2}$;
\item[Step 2]\label{it:alg2step2} by results in Section \ref{sec:gaps} we determine all the integers of $R_d$ that are certainly gaps;
\item[Step 3]\label{it:alg2step3} using algorithm $\textsc{genusSearch}$ (Algorithm \ref{alg:genusTreeSearch}) we investigate the remaining integers.
\end{description}

\begin{algorithm}[!ht]
\caption{\label{alg:ACMgenera} The algorithm for determining the aCM genera of curves with a given degree. A trial version of this algorithm is available at \href{http://www.paololella.it/HSC/Finite_O-sequences_and_ACM_genus.html}{\small \tt http://www.paololella.it/HSC/} \href{http://www.paololella.it/HSC/Finite_O-sequences_and_ACM_genus.html}{\small \tt Finite\_O-sequences\_and\_ACM\_genus.html}}
\begin{algorithmic}[1]
\Procedure{ACMgenera}{$d$}
\Require $d$, a positive integer.
\Ensure the list of all possible aCM genera of a curve of degree $d$.
\State $\textsf{genera} := \{\textnormal{genera determined applying recursively Lemma \ref{gradi successivi}}\}$;
\State $\textsf{gaps} := \{\textnormal{gaps determined applying Proposition \ref{easy gaps} and Proposition \ref{2}}\}$; 
\State $\textsf{undecided} := \left\{0,\ldots,\binom{d-1}{2}\right\} \setminus \big( \textsf{genera} \cup \textsf{gaps}\big)$;
\For{$s=2,\ldots,d-3$}
\State $g:= \min (\textsf{undecided})$;
\While{$g \leq \textsc{upperBound}(R_d^s)$}
\If{$g < \textsc{lowerBound}(R_d^s)$}
\State $\textsc{remove}(g,\textsf{undecided})$;
\State $\textsf{gaps} = \textsf{gaps} \cup \{g\}$;
\Else 
\State $\textsf{searching} := \Call{genusSearch}{g,\mathcal{T}_d^s}$;
\If{$\textsf{searching} \neq \emptyset$}
\State $\textsc{remove}(g,\textsf{undecided})$;
\State $\textsf{genera} = \textsf{genera} \cup \{g\}$;
\EndIf
\EndIf
\State $g = \textsc{next}(g,\textsf{undecided})$;
\EndWhile
\EndFor
\State \Return $\textsf{genera}$;
\EndProcedure
\end{algorithmic}
\end{algorithm}

\begin{table}[!ht]
\caption{\label{tab:generaDistribution} In this table, we report some numerical information about the integers in $G_d$ up to degree $250$. The first column contains the number and the percentage of values in $R_d$ which are aCM genera by an application of Lemma \ref{gradi successivi} (without computing the O-sequences); in the second column, the number and the percentage of gaps determined applying Proposition \ref{easy gaps} and Proposition \ref{2}; in the third column, the number and the percentage of values of $R_d$ for which we have to use the procedure \textsc{genusSearch} to decide whether they are aCM genera; in the last column, the cardinality of $G_d$ and its percentage with respect to $\vert R_d\vert$.}
\begin{center}
{\tiny
\begin{tabular}{c | c | c | c | c }	
{\small $\phantom{\big\vert}d\phantom{\big\vert}$} & {\small Certain genera} & {\small Certain gaps} & {\small Undecided values} & {\small $\vert G_d\vert$}\\ 
\hline
{\small$\phantom{\big\vert}25\phantom{\vert}$} & $ 176\ (63.77\%)$ &
			$88\ (31.88\%)$&
				$13\ (4.71\%)$& 
					$187\ (67.75\%)$\\ 
\hline
{\small$\phantom{\big\vert}50\phantom{\vert}$} & $ 835\ (71.00\%)$ &
			$289\ (24.57\%)$&
				$53\ (4.51\%)$& 
					$870\ (73.98\%)$\\ 
\hline
{\small$\phantom{\big\vert}75\phantom{\vert}$} & $ 2033\ (75.27\%)$ &
			$558\ (20.66\%)$&
				$111\ (4.11\%)$& 
					$2099\ (77.71\%)$\\ 
\hline
{\small$\phantom{\big\vert}100\phantom{\vert}$} & $ 3798\ (78.29\%)$ &
			$879\ (18.12\%)$&
				$175\ (3.61\%)$& 
					$3894\ (80.27\%)$\\ 
\hline
{\small$\phantom{\big\vert}125\phantom{\vert}$} & $ 6129\ (80.37\%)$ &
			$1244\ (16.31\%)$&
				$254\ (3.33\%)$& 
					$6261\ (82.10\%)$\\ 
\hline
{\small$\phantom{\big\vert}150\phantom{\vert}$} & $ 9040\ (81.99\%)$ &
			$1653\ (14.99\%)$&
				$334\ (3.02\%)$& 
					$9207\ (83.50\%)$\\ 
\hline
{\small$\phantom{\big\vert}175\phantom{\vert}$} & $ 12528\ (83.24\%)$ &
			$2094\ (13.91\%)$&
				$430\ (2.86\%)$& 
					$12734\ (84.61\%)$\\ 
\hline
{\small$\phantom{\big\vert}200\phantom{\vert}$} & $ 16610\ (84.31\%)$ &
			$2574\ (13.07\%)$&
				$518\ (2.63\%)$& 
					$16854\ (85.55\%)$\\ 
\hline
{\small$\phantom{\big\vert}225\phantom{\vert}$} & $ 21276\ (85.19\%)$ &
			$3084\ (12.35\%)$&
				$617\ (2.47\%)$& 
					$21560\ (86.32\%)$\\ 
\hline
{\small$\phantom{\big\vert}250\phantom{\big\vert}$} & $26530\ (85.92\%)$&
			$3623\ (11.73\%)$&
				$724\ (2.34\%)$& 
					$26856\ (86.98\%)$\\ 
\end{tabular}
}
\end{center}
\end{table}

\begin{table}[!ht]
\caption{\label{tab:computationTime} In this table, we report the results of a test of Algorithm \ref{alg:ACMgenera} up to degree $250$. The first three columns contain the elapsed time (in milliseconds) for \textbf{Step 1}, \textbf{Step 2} and \textbf{Step 3} of Algorithm \ref{alg:ACMgenera}. In the fourth column, there is the total time  for the execution (Step 1 + Step 2 + Step 3). The last column contains the time required for determining the set $G_d$ by performing a complete visit of the tree $\mathcal{T}_d$ (even for $d=75$, we obtain an Out Of Memory Error). The algorithms are implemented in the \texttt{Java} language and have been run on a MacBook Pro with an Intel Core 2 Duo 2.4 GHz processor.}
\begin{center}
{\tiny
\begin{tabular}{c | c | c | c | c | c }	
{\small $d$} & {\small \quad Step 1\quad} & {\small \quad Step 2\quad} & {\small \quad Step 3\quad} & {\small Algorithm \ref{alg:ACMgenera}} & {\small Visit $\mathcal{T}_d$ }\\ 
\hline
{\small$\phantom{\big\vert}25\phantom{\vert}$} & $ 37.336\,\text{ms} $ &
			$ 0.164\,\text{ms} $&
				$ 38.594\,\text{ms} $& $76.094\, \text{ms}$ & 
					$210.769\, \text{ms}$\\ 
\hline
{\small$\phantom{\big\vert}50\phantom{\vert}$} & $ 82.774\,\text{ms} $ &
			$0.208\,\text{ms} $&
				$212.868\,\text{ms} $& $295.850\, \text{ms}$ &
					$15155.87\, \text{ms}$\\ 
\hline
{\small$\phantom{\big\vert}75\phantom{\vert}$} & $ 21.734\,\text{ms} $ &
			$0.155\,\text{ms} $&
				$ 458.117\,\text{ms} $& $480.006\, \text{ms}$ & O.O.M. \\ 
\hline
{\small$\phantom{\big\vert}100\phantom{\vert}$} & $ 47.529\,\text{ms} $ &
			$ 0.103\,\text{ms} $&
				$ 1390.027\,\text{ms} $& $1437.659\, \text{ms}$ & O.O.M.\\ 
\hline
{\small$\phantom{\big\vert}125\phantom{\vert}$} & $ 104.683\,\text{ms} $ &
			$ 0.279\,\text{ms} $&
				$ 4684.598\,\text{ms} $& $4789.56\, \text{ms}$ &  O.O.M.\\ 
\hline
{\small$\phantom{\big\vert}150\phantom{\vert}$} & $ 207.936\,\text{ms} $ &
			$ 0.183\,\text{ms} $&
				$ 12610.461\,\text{ms} $& $12818.58\, \text{ms}$ & O.O.M.\\ 
\hline
{\small$\phantom{\big\vert}175\phantom{\vert}$} & $ 546.818\,\text{ms} $ &
			$0.227\,\text{ms} $&
				$ 37518.036\,\text{ms} $& $38065.081\, \text{ms}$ & O.O.M.\\ 
\hline
{\small$\phantom{\big\vert}200\phantom{\vert}$} & $665.378\,\text{ms} $ &
			$0.364\,\text{ms} $&
				$ 73552.564\,\text{ms} $& $74218.306\, \text{ms}$ & O.O.M.\\ 
\hline
{\small$\phantom{\big\vert}225\phantom{\vert}$} & $ 922.599\,\text{ms} $ &
			$0.36\,\text{ms} $&
				$ 169042.878\,\text{ms} $& $169965.837\, \text{ms}$ & O.O.M.\\ 
\hline
{\small$\phantom{\big\vert}250\phantom{\big\vert}$} & $1395.378\,\text{ms}$&
			$ 0.179\,\text{ms}$&
				$ 359836.564\,\text{ms}$& $361232.121\, \text{ms}$ & O.O.M.\\ 
\end{tabular}
}
\end{center}
\end{table}


\section{An application: Castelnuovo-Mumford regularity of curves with Cohen-Macaulay postulation}\label{sec:CMreg}

In this section, we show how the search algorithm of aCM genera (Algorithm \ref{alg:genusTreeSearch}) allows us to detect the minimal Castelnuovo-Mumford regularity $m_{d,g}^{\textnormal{aCM}}$ of a curve with Cohen-Macaulay postulation, given its degree $d$ and genus $g$. Moreover, by the Example \ref{even aCM} we give a negative answer to a question posed in \cite[Remark 2.5]{CDG}. A complete solution to the problem of detecting the minimal Castelnuovo-Mumford regularity of a scheme with a given Hilbert polynomial is described in \cite{CMreg}. 

Denoting by $\rho$ the {\em regularity} of a Hilbert function, i.e.~the minimal degree from which the Hilbert function and the Hilbert polynomial coincide, we can state the following:
\begin{proposition}\label{regACM}
\[
m_{d,g}^{\textnormal{aCM}} = \min \left\{\rho\ \Bigg\vert\ \begin{array}{l}\rho \text{ is the regularity of an aCM postulation}\\ \text{with Hilbert polynomial } dt+1-g \end{array}\right\} + 2
\]
\end{proposition}

\begin{proof}
Let $f$ be an aCM postulation with Hilbert polynomial $dt+1-g$ and regularity $\rho$. Then, the minimal possible Castelnuovo-Mumford regularity of a curve with Hilbert function $f$ is $\rho+2$. As a matter of fact, by \cite[Proposition 2.4]{CDG} this regularity is strictly greater than $\rho+1$ and if the curve is aCM, it is exactly $\rho+2$.
\end{proof}

By Proposition \ref{regACM}, the value of $m_{d,g}^{\textnormal{aCM}}$ is determined by applying Algorithm \ref{alg:genusTreeSearch} in order to find an O-sequence $\mathsf{h}$ of multiplicity $d$ and $g(\mathsf{h}) = g$ with the shortest possible length. Notice that if the length of $\mathsf{h}$ is $s$, then the regularity of $\Sigma^2 \mathsf{h}$ is $s-2$. Thus, we can rewrite the statement in Proposition \ref{regACM} as
\[
m_{d,g}^{\textnormal{aCM}} = \min \left\{ s \ \Bigg\vert\ \begin{array}{l}s \text{ is the length of an O-sequence $\mathsf{h}$}\\ \text{with multiplicity } d \text{ and } g(\mathsf{h}) = g \end{array}\right\}.
\]

\begin{example}
Let us consider the curves of degree $d=15$ and genus $g=32$. There are four O-sequences of multiplicity $d$ corresponding to aCM curves of genus $g$:
\[
\begin{array}{llll}
\mathsf{h}_1 = (1,4,3,2,1,1,1,1,1), &&& \mathsf{h}_3  = (1,2,3,4,2,1,1,1),\\
\mathsf{h}_2 = (1,3,3,2,2,2,1,1), &&& \mathsf{h}_4  = (1,3,5,1,1,1,1,1,1).
\end{array}
\]
Hence, the minimal Castelnuovo-Mumford regularity $m_{d,g}^{\textnormal{aCM}}$ is $8$. Applying the results of \cite{CMreg}\hfill (see\hfill \href{http://www.paololella.it/HSC/Minimal_Hilbert_Functions_and_CM_regularity.html}{\small \tt http://www.paololella.it/HSC/Minimal\_Hilbert\_Functions\_and\_CM\_}\\ \href{http://www.paololella.it/HSC/Minimal_Hilbert_Functions_and_CM_regularity.html}{\small \tt regularity.html}), we notice that the minimal Castelnuovo-Mumford regularity of any projective scheme with Hilbert polynomial $p(t)=15t-31$ is $7$. 
\end{example}

More generally, in the case of an aCM function $f$ with regularity $\rho$ and Hilbert polynomial with {\em odd} degree, we have that the minimal possible Castelnuovo-Mumford regularity of a scheme $X$ with $H_X=f$ is strictly greater than $\rho+1$ (see \cite[Proposition 2.4]{CDG}). If the degree of the Hilbert polynomial is {\em even}, an analogous result does not hold, as the following example shows.

\begin{example}\label{even aCM}
The following strongly-stable ideal 
\[
\begin{split}
I = (&x_6^2, x_5x_6, x_5^2, x_4x_5, x_3x_5, x_2x_5, x_1x_5, x_4^2x_6, x_3x_4x_6, x_2x_4x_6, x_1x_4x_6, x_3^2x_6,x_2x_3x_6, \\ 
&x_1x_3x_6, x_2^3x_6, x_1x_2^2x_6, x_1^2x_2x_6, x_4^4, x_3x_4^3, x_2x_4^3, x_1^4x_6, x_3^3x_4^2, x_3^4x_4, x_3^5) \subset K[x_0,\ldots,x_6],
\end{split}
 \]
where $x_0<x_1<\cdots <x_6$, defines a non-aCM surface $X\subset \mathbb P^6$ with the aCM postulation $H_X=(1,7,21,44,\ldots,6t^2-10t+21,\ldots)$ of regularity $\rho=4$ and the Castelnuovo-Mumford regularity of $X$ is $5=\rho+1$. 
\end{example}


\providecommand{\bysame}{\leavevmode\hbox to3em{\hrulefill}\thinspace}
\providecommand{\MR}{\relax\ifhmode\unskip\space\fi MR }
\providecommand{\MRhref}[2]{%
  \href{http://www.ams.org/mathscinet-getitem?mr=#1}{#2}
}
\providecommand{\href}[2]{#2}

\end{document}